\documentclass[10pt]{jotart}

\usepackage{amsthm}

\usepackage{amssymb}

\usepackage[english]{babel}
\usepackage{amsmath}
\usepackage[utf8]{inputenc}
\usepackage{eufrak}
\usepackage{amscd}
\usepackage{mathrsfs}
\usepackage{enumerate}
\usepackage{yhmath}
\usepackage[pagewise]{lineno}%\linenumbers

\usepackage{hyperref}

\theoremstyle{proclaim}
\newtheorem{theorem}{Theorem}[section]
\newtheorem{corollary}[theorem]{Corollary}
\newtheorem{lemma}[theorem]{Lemma}
\newtheorem{proposition}[theorem]{Proposition}

\theoremstyle{statement}
\newtheorem{definition}[theorem]{Definition}

\newtheorem{example}[theorem]{Example}
\newtheorem{remark}[theorem]{Remark}
\newtheorem{question}[theorem]{Question}

\numberwithin{equation}{section}

\providecommand{\AMS}{$\mathcal{A}$\kern-.1667em%
\lower.25em\hbox{$\mathcal{M}$}\kern-.125em$\mathcal{S}$}

\begin{document}

%\date{January 28, 2022}

\title[Irreducible representations, Lebesgue decomposition]{Finite dimensional irreducible representations and the uniqueness of the Lebesgue decomposition of positive functionals}


\author[Szűcs, Takács]{Zsolt Sz\H{u}cs, Balázs Takács}
%\ead{szucszs@math.bme.hu}
\address{Department of Differential Equations, Institute of Mathematics, Budapest University of Technology and Economics, M\H{u}egyetem rkp. 3., H-1111 Budapest, Hungary}

\email{szucszs@math.bme.hu}

\address{Department of Analysis, Institute of Mathematics, Budapest University of Technology and Economics, M\H{u}egyetem rkp. 3., H-1111 Budapest, Hungary;  Department of Applied Quantitative Methods,  Budapest Business School, Buzogány utca 10-12., H-1149 Budapest, Hungary}

\email{takacsb@math.bme.hu}


\begin{abstract}
For an arbitrary complex $^*$-algebra $A$, we prove that every topologically irreducible $^*$-representation of $A$ on a Hilbert space is finite dimensional precisely when the Lebesgue decomposition of representable positive functionals over $A$ is unique.
In particular, the uniqueness of the Lebesgue decomposition of positive functionals over the $L^1$-algebras of locally compact groups provides a new characterization of Moore groups.
\end{abstract}



\maketitle

\emph{\textbf{To appear in Journal of Operator Theory.}}

\section*{INTRODUCTION}

The classical measure-theoretic Lebesgue decomposition is a fundamental theorem for every mathematician. 
It states that, if $\mu$ and $\nu$ are finite measures on the measurable space $(\Omega, \mathscr{F})$, then there exist measures $\mu_r$ and $\mu_s$ on $(\Omega, \mathscr{F})$ such that $\mu_r$ is absolutely continuous with respect to $\nu$, $\mu_s$ and $\nu$ are singular to each other and $\mu$ is decomposed as the sum $\mu=\mu_r+\mu_s$. A simple but very important property holds for this decomposition: it is unique in the manner that if $\mu_r'$, $\mu_s'$ are measures on $(\Omega, \mathscr{F})$ such that
$\mu_r'$ is absolutely continuous with respect to $\nu$, $\mu_s'$ and $\nu$ are singular, $\mu=\mu_r'+\mu_s'$, then $\mu_r'=\mu_r$ and $\mu_s'=\mu_s$.

J. von Neumann's proof of the Radon-Nikodym theorem involving Hilbert space techniques (\cite{vN}, Lemma 3.2.3 and Theorem VII) inspired B. Simon (\cite{Si}) to investigate the closability of non-negative quadratic forms on Hilbert spaces.
More generally (see Appendix B in \cite{GK}), dropping the assumption of the completeness, 
one can say that a (non-negative quadratic) form $\mathbf{f}$ on the complex vector space $A$ is \emph{closable} with respect to another form $\mathbf{g}$ on $A$, if the property
\[
\left(\mathbf{f}(a_n-a_m, a_n-a_m)\to0\ \wedge\ \mathbf{g}(a_n, a_n)\to0\right)\ \Rightarrow\ \mathbf{f}(a_n, a_n)\to0
\]
is true for any sequence $(a_n)_{n\in\mathbb{N}}$ in $A$.
In their remarkable paper (\cite{GK}, Theorem 7.7), A. Gheondea and A. Ş. Kavruk proved that $\mathbf{f}$ is closable with respect to $\mathbf{g}$ if and only if there exist a sequence $(\mathbf{f}_n)_{n\in\mathbb{N}}$ of forms on $A$ and a sequence $(\alpha_n)_{n\in\mathbb{N}}$ of positive numbers such that
\[
\mathbf{f}_n\leq \mathbf{f}_{n+1},\ \mathbf{f}_n\leq\alpha_n\mathbf{g}\ (\forall n\in\mathbb{N}),\ \mathbf{f}(a,b)=\lim_{n\to+\infty}\mathbf{f}_n(a,b)\ (\forall a,b\in A).
\]
If the latter condition is fulfilled, then we say that $\mathbf{f}$ is \emph{absolutely continuous with respect to $\mathbf{g}$}.

The terms above are supported by the following observation. 
Let $\mu$ and $\nu$ be finite measures on the measurable space $(\Omega, \mathscr{F})$ and let 
$A$ be the linear span of the characteristic functions of the sets in $\mathscr{F}$ over the complex numbers. 
Then the formulas 
\[
\mathbf{f}(a,b):=\int_{\Omega} a\overline{b}\ \mathrm{d}\mu;\ \mathbf{g}(a,b):=\int_{\Omega} a\overline{b}\ \mathrm{d}\nu \ \ \ \ (a.b\in A)
\]
define forms on the vector space $A$, moreover $\mathbf{f}$ is absolutely continuous with respect to $\mathbf{g}$ iff $\mu$ is absolutely continuous with respect to $\nu$. 
A similar observation is true, if we define singularity for forms as follows: $\mathbf{f}$ and $\mathbf{g}$ are \emph{singular}, if for every form $\mathbf{p}'$ the inequalities $\mathbf{p}'\leq\mathbf{f}$ and $\mathbf{p}'\leq\mathbf{g}$ imply that $\mathbf{p}'=0$.

For a further important generalization, notice that the set $A$ of simple functions is not just a vector space. 
It is a complex $^*$-algebra with the pointwise operations, furthermore the formulas
\begin{equation}\label{megf}
f(a):=\int_{\Omega}a\ \mathrm{d}\mu;\ g(a):=\int_{\Omega}a\ \mathrm{d}\nu \ \ \ \ (a\in A)
\end{equation}
define positive linear functionals on $A$.
If we introduce the notions of absolute continuity and singularity for these objects by the form-related definitions with $\mathbf{f}(a,b)=f(b^*a)$, then one can show that 
$f$ is absolutely continuous with respect to $g$ exactly when $\mu$ is absolutely continuous with respect to $\nu$. 
Moreover, $f$ and $g$ are singular precisely when $\mu$ and $\nu$ are singular (e.g., \cite{SzLebalg}, Lemma 4.1; see also Example 1 in \cite{G}). 
Since the definitions do not involve the commutativity, these concepts can be investigated over arbitrary $^*$-algebras. 
This leads to the Lebesgue decomposition theory of positive functionals defined on $^*$-algebras (see subsection \ref{felbont} below).

In the past decades, several non-commutative generalizations of the Le\-bes\-gue-Radon-Nikodym theory have been appeared in the mathematical literature.
Without being exhaustive, we mention here a few directions and papers (for further reference, see the introductions and the bibliographies of \cite{GK} and \cite{HSdS}).
In the case of 
\begin{itemize}
\item non-negative quadratic forms on complex vector spaces: B. Simon \cite{Si}; A. Gheondea, A. Ş. Kavruk \cite{GK}; S. Hassi, Z. Sebestyén and H. de Snoo \cite{HSdS}.

\item positive operators on Hilbert spaces: T. Ando \cite{Ando}; A. Gheondea, A. Ş. Kavruk \cite{GK}.

\item operator valued completely positive maps on $C^*$-algebras: W. B. Arveson \cite{Arveson}; K. R. Parthasarathy \cite{Par}; A. Gheondea, A. Ş. Kavruk \cite{GK}.

\item positive functionals on $^*$-algebras: S. Sakai \cite{Sa} (1.24); S. P. Gudder \cite{G}; H. Kosaki \cite{Ko}; Zs. Szűcs \cite{SzLebalg}.
\end{itemize}

Using Ando's famous Lebesgue theory for positive operators (\cite{Ando}; \cite{GK}, Appendix A), Gheondea and Kavruk (\cite{GK}, Theorem 7.8) proved the following noteworthy form related decomposition.
Every form $\mathbf{f}$ on a complex vector space $A$ can be decomposed as a sum 
\begin{equation}\label{formafelb}
\mathbf{f}=\mathbf{f}_r+\mathbf{f}_s
\end{equation}
with respect to another form $\mathbf{g}$ on $A$, where $\mathbf{f}_r$ is absolutely continuous with respect to $\mathbf{g}$, while $\mathbf{f}_s$ and $\mathbf{g}$ are singular.
The form $\mathbf{f}_r$ is the greatest among all of the forms $\mathbf{f}_0$ on $A$ such that $\mathbf{f}_0\leq \mathbf{f}$ and $\mathbf{f}_0$ is absolutely continuous with respect to $\mathbf{g}$.
This result implies the form decomposition of Simon (\cite{Si}, Theorem 2.5).
(Note here that, in \cite{HSdS}, Theorems 2.11 and 3.9 Hassi, Sebestyén and de Snoo obtained the decomposition \eqref{formafelb} without Ando's results, but \eqref{formafelb} appeared first in \cite{GK}.)

By Ando's results, another decomposition was obtained in \cite{GK}, Theorem 3.1, namely, Parthasarathy's decomposition  for completely positive maps (\cite{Par}, page 44). 
But, similar to the case of positive operators and forms, the maximality of the absolutely continuous part was also proved in this theorem.
Moreover, Appendix C  in \cite{GK} shows that Ando's theory can derived from the form and the completely positive map case.
Thus, \cite{GK} makes a remarkable connection between the different settings of the non-commutative Lebesgue-Radon-Nikodym theory.  

For positive functionals, Gudder (\cite{G}, Corollary 3) proved a Lebesgue decomposition theorem on unital Banach $^*$-algebras. 
In \cite{Ko}, Theorem 3.5, H. Kosaki gave a similar decomposition for normal states on $\sigma$-finite von Neumann algebras. As a common generalization, using the form decomposition \eqref{formafelb}, the first author of the present paper introduced a Lebesgue decomposition for representable positive functionals on arbitrary $^*$-algebras (\cite{SzLebalg}, Corollary 3.2; see Theorem \ref{Lebesgue} and Remark \ref{why} below).
Note that the $C^*$-algebra version is a special case of the decomposition related to completely positive maps.

Turning to the subject of the present paper, we examine the question of the uniqueness.
In contrast to the measure-theoretic decomposition above, the decompositions in the non-commutative settings are not unique in general.
For instance, in the form case this means that there exist forms $\mathbf{f}$, $\mathbf{g}$ on the vector space $A$ such that $\mathbf{f}=\mathbf{f}_r'+\mathbf{f}_s'$, where the form $\mathbf{f}_r'$ is absolutely continuous with respect to $\mathbf{g}$, the forms $\mathbf{f}_s'$ and $\mathbf{g}$ are singular, moreover $\mathbf{f}_r'\neq\mathbf{f}_r$ and $\mathbf{f}_s'\neq\mathbf{f}_s$.
For positive operators, Ando (\cite{Ando}, Theorem 6; \cite{GK}, Theorem 6.5) showed that the Lebesgue decomposition of a positive operator $S_1$ with respect to $S_2$ is unique iff $S_2$ uniformly dominates the absolutely continuous part of $S_1$. 
By Corollary 7 in \cite{Ando} (see also Corollary 6.6 in \cite{GK}), one can easily find operators for which the decomposition is not unique. 

Using Ando's attractive characterization statement on the uniqueness, Ghe\-ondea and Kavruk (\cite{GK}, Theorem 7.12 (1)) proved that a similar result holds in the case of the form decomposition.
Namely, the decomposition $\mathbf{f}=\mathbf{f}_r+\mathbf{f}_s$ in \eqref{formafelb} is unique if and only if there is an $\alpha\geq0$ such that
\[ 
\mathbf{f}_r(a,a)\leq\alpha\mathbf{g}(a,a)
\]
is true for any $a\in A$. 
This theorem implies that the decomposition of Simon is not unique in general as well. 
Moreover, Corollary 7.12 (2) in \cite{GK} also gives equivalent properties for the uniqueness of the decomposition with respect to $\mathbf{g}$. 

The case of completely positive maps is complicated.
Similar to the form case (Corollary 7.12 (2) in \cite{GK}), Proposition 3.8 in \cite{GK} presents conditions on the uniqueness, but they are not necessary properties.
Example 3.9 in \cite{GK} and the discussion before it explains the reasons (see also Example 5.14 in \cite{SzAbs}).

For positive functionals, a major unanswered question remained:
Which $^*$-algebras have the property that the Lebesgue decomposition of representable positive functionals over the $^*$-algebra is unique? Can one obtain a characterizing assertion? Until now, the solution of this problem had only partial results, for example, uniqueness over commutative $^*$-algebras, non-uniqueness over properly infinite von Neumann algebras (see the discussion after Definition \ref{defunique}).
The main subject of the paper is to settle this uniqueness problem in full generality. 
As a solution, we give a concrete characterizing criterion for the uniqueness via the representation theory of the underlying $^*$-algebra.
This is our main result, Theorem \ref{main}.

%The mathematical literature contains various remarkable generalizations of this result. A common starting point for these are based upon the following observations.

\section{PRELIMINARIES}\label{prelims}

To make our aim more clear, we introduce the fundamental definitions and theorems related to the Lebesgue decomposition theory in the context of positive functionals defined on $^*$-algebras. 
In the first place, we have to recall some indispensable concepts and results of the representation theory of $^*$-algebras on Hilbert spaces.

\subsection{Some generalities} The following well-known facts can be found in the textbooks which are listed in the bibliography (\cite{Bla2}, \cite{Dixmier}, \cite{DvN}, \cite{KR2}, \cite{Palmer2}, \cite{Pedersen}, \cite{Sa}, \cite{Ta1}). 
We use them further without any reference.

Let us fix a $^*$-algebra $A$, that is, an algebra over the complex numbers $\mathbb{C}$, endowed with an involution $^*:A\to A$.
By a \emph{($^*$-)representation} of $A$ we always mean a $^*$-homomorphism $\pi:A\to\mathscr{B}(\mathcal{H})$, where $\mathcal{H}$ is a complex Hilbert space with inner product $(\cdot\!\mid\!\cdot)$ 
and $\mathscr{B}(\mathcal{H})$ is the $^*$-algebra of bounded linear operators on $\mathcal{H}$. 
A linear subspace $\mathcal{H}_0$ of $\mathcal{H}$ is \emph{$\pi$-invariant}, if
\[
\pi\left<A\right>\mathcal{H}_0:=\{\pi(a)\xi\!\mid\!a\in A,\ \xi\in\mathcal{H}_0\}\subseteq\mathcal{H}_0.
\]

We say that a non-zero representation $\pi$ of $A$ on the Hilbert space $\mathcal{H}$ is 
\begin{itemize}
	\item \emph{non-degenerate} or \emph{essential}, if the span of 
	\[
	\pi\left<A\right>\mathcal{H}:=\{\pi(a)\xi\!\mid\!a\in A,\ \xi\in\mathcal{H}\}
	\]
	is dense in $\mathcal{H}$.
	
	\item \emph{cyclic}, if there exists a $\xi\in\mathcal{H}$ such that the subspace
	\[
	\pi\left<A\right>\xi:=\{\pi(a)\xi\!\mid\!a\in A\}
	\]
	is dense in $\mathcal{H}$. We call such a vector \emph{$\pi$-cyclic vector}.
	
	\item \emph{faithful}, if the mapping $\pi$ is injective.
	
	\item \emph{topologically irreducible}, if the closed $\pi$-invariant subspaces are only $\{0\}$ and $\mathcal{H}$. 
	This is equivalent to the property that every $\xi\in\mathcal{H}\setminus\{0\}$ is a $\pi$-cyclic vector.
	
	\item \emph{algebraically irreducible}, if the $\pi$-invariant subspaces are only $\{0\}$ and $\mathcal{H}$. 
	This is equivalent to the property that $\pi\left<A\right>\xi=\mathcal{H}$ holds for every $\xi\in\mathcal{H}\setminus\{0\}$.
\end{itemize}

\begin{remark}\label{trivial}
According to Definition 9.2.1 in \cite{Palmer2} and Definitions 4.1.1, 4.2.1 in \cite{Palmer1}, our convention is that the terms \emph{topologically/algebraically irreducible representations} always refer to non-zero representations. 
\end{remark}

If $\mathscr{S}\subseteq\mathscr{B}(\mathcal{H})$ is an arbitrary set of operators, then the \emph{commutant} of $\mathscr{S}$ is
\[
\mathscr{S}':=\{U\in\mathscr{B}(\mathcal{H})\!\mid\!\forall S\in\mathscr{S}:\ US=SU \}. 
\]

A linear functional $f:A\to\mathbb{C}$  is \emph{positive}, if $f(a^*a)\geq0$ for any $a\in A$ (in notation: $f\geq0$). 
For positive functionals $f$ and $g$ defined on $A$, we write $f\leq g$, if the functional $g-f$ is positive. 
It is obvious that the sum of positive functionals and non-negative multiples of positive functionals are positive.
The notation $L_f$ stands for the \emph{left kernel} of $f$, i.e., the left ideal
\[
L_f:=\{a\in A\!\mid\!f(a^*a)=0\}.
\]
We call $f$ \emph{faithful}, if $L_f=\{0\}$. 

A positive functional $f$ is said to be \emph{representable}, if there exist a Hilbert space $\mathcal{H}$, a cyclic representation $\pi:A\to\mathscr{B}(\mathcal{H})$ and a $\pi$-cyclic vector $\xi\in\mathcal{H}$ such that
\[
f(a)=(\pi(a)\xi\!\mid\!\xi)\ \ \ \ (\forall a\in A).
\] 
We note that the sum and non-negative multiples of representable positive functionals are representable.
%Moreover, if $f,g:A\to\mathbb{C}$ are positive functionals, $g$ is representable, then $f\leq g$ implies the representability of $f$.
Moreover, if $(f_n)_{n\in\mathbb{N}}$ is an increasing sequence of representable positive functionals on $A$ (that is, $f_n\leq f_{n+1}$ for all $n\in\mathbb{N}$) which is bounded by a representable positive functional on $A$, then $(f_n)_{n\in\mathbb{N}}$ converges pointwise to a representable positive functional (\cite{SzAbs}, Remark 2.14). The notation $\sup_{n\in\mathbb{N}}f_n$ stands for this functional. 

\begin{remark}\label{trivial2}
We record here that there exists a topologically irreducible (hence non-zero) representation of $A$ if and only if there is a non-zero representable positive functional on $A$. 
This follows from Theorems 9.6.4 and 9.6.6 (b) in \cite{Palmer2} (see Definitions 9.4.21 and 9.6.3 in \cite{Palmer2}).   
\end{remark}

Let $M$ be a von Neumann algebra, that is, a $C^*$-algebra of operators on a Hilbert space $\mathcal{H}$ which is closed in the strong operator topology and contains the identity. 
Equivalently, $M=M''$ in $\mathscr{B}(\mathcal{H})$.
A positive functional $\mathfrak{f}:M\to\mathbb{C}$ is said to be \emph{normal}, if whenever $(x_i)_{i\in I}$ is a norm-bounded increasing net in the set of positive elements of $M$ with $x=\sup_{i\in I}x_i$, then $\mathfrak{f}(x)=\sup_{i\in I}\mathfrak{f}(x_i)$. 
It is obvious that the sum and non-negative multiples of normal positive functionals are normal. Moreover, if $\mathfrak{f},\mathfrak{g}:M\to\mathbb{C}$ are positive functionals and $\mathfrak{g}$ is normal, then $\mathfrak{f}\leq \mathfrak{g}$ implies the normality of $\mathfrak{f}$.
The von Neumann algebra $M$ is \emph{$\sigma$-finite}, if there exists a faithful normal positive functional on $M$.

\subsection{Lebesgue decomposition of positive functionals}\label{felbont} 
This part contains the precise concepts of the theory, including the decomposition theorem for representable positive functionals on $^*$-algebras (Theorem \ref{Lebesgue}). 
We discuss the subject of the paper below, namely, the problem of the u\-ni\-que\-ness of the Lebesgue decomposition, and we present the solution of this problem as the main result of this article (Theorem \ref{main}).

The notion of absolute continuity has many equivalent formulations, e.g., in the case of
\begin{itemize}
\item forms: \cite{Si}; \cite{GK}, Appendix B; \cite{HSdS}, 2.5 and Theorem 3.8.
\item positive operators: \cite{Ando}, Introduction, Lemma 1 and Section 3; \cite{GK}, Appendix A.
\item completely positive maps: \cite{Par}, page 48; \cite{GK}, II and Corollary 3.5.
\item positive functionals: Theorem 1 and Corollary 2 in \cite{G}; Theorem 2.2 in \cite{Ko}; Theorem 2.15 in \cite{SzAbs}.
\end{itemize}
We only need the following version.

\begin{definition}\label{absdef}
	Let $A$ be a $^*$-algebra and let $f,g:A\to\mathbb{C}$ be representable positive functionals on $A$. 
We say that $f$ is \emph{absolutely continuous with respect to $g$}, if there exist a sequence $(f_n)_{n\in\mathbb{N}}$ of representable positive functionals on $A$ and a sequence $(\alpha_n)_{n\in\mathbb{N}}$ of positive numbers such that
		\[
		f_n\leq f_{n+1},\ f_n\leq\alpha_ng\ (\forall n\in\mathbb{N}),\ f=\sup_{n\in\mathbb{N}}f_n.
		\]
In notation, we use $f\ll g$.
\end{definition}

The concept of singularity also has various equivalent formulations, e.g., in the case of
\begin{itemize}
\item forms: \cite{GK}, page 24 and Corollary 7.11; \cite{HSdS}, 2.5 and Corollary 3.11.
\item positive operators: \cite{Ando}, Introduction and Corollary 3; \cite{GK}, Appendix A.
\item completely positive maps: \cite{Arveson}, Corollary 1.4.4.; \cite{GK}, Corollary 3.6.
\item positive functionals: page 146 in \cite{G}; Theorem 8.1 in \cite{Ko}; Theorem 3 in \cite{SzSing}.
\end{itemize}
The following is the suitable version for us.

\begin{definition}\label{singdef}
	Let $A$ be a $^*$-algebra. 
	We say that the representable positive functionals $f,g:A\to\mathbb{C}$ are \emph{singular} (to each other), if for every representable positive functional $p'$ on $A$ the conditions
	\[
	p'\leq f,\ p'\leq g 
	\]
	imply that $p'=0$. In notation, we use $f\perp g$ for this property.
\end{definition}

\begin{remark}\label{altul}
	We record here some easy consequences of the definitions of absolute continuity and singularity. 
	Let $A$ be a $^*$-algebra, $f$ and $g$ are representable positive functionals on $A$. 
	\begin{itemize}
		\item Let $B$ be a $^*$-subalgebra of $A$. 
		If $f$ is absolutely continuous with respect to $g$, then the same is true for the restrictions $f \!\mid\!_B$ and $g\!\mid\!_B$. (These functionals are representable as well.)
		
		\item The implications
		\begin{equation}\label{WDtul}
		f\leq g\ \Rightarrow\ f\ll g\ \Rightarrow\ L_g\subseteq L_f
		\end{equation}
		are true in general.
		
		\item If $f$ and $g$ are singular and $f$ is absolutely continuous with respect to $g$ at the same time, then the functional $f$ is zero. Indeed, by singularity, every $f_n$ is zero in Definition \ref{absdef}. 
	\end{itemize}
\end{remark}

The following general Lebesgue decomposition theorem for representable positive functionals on $^*$-algebras can be found in the first author's papers (\cite{SzLebalg}, Corollary 3.2, or Corollary 2.17 and Proposition 1.13  in \cite{SzAbs}).
We also note that the original measure-theoretic Lebesgue decomposition is an easy consequence by the associations in \eqref{megf}  
(e.g., \cite{SzLebalg}, Theorem 4.2).

\begin{theorem}\label{Lebesgue}
	Let $A$ be a $^*$-algebra and let $f,g$ be representable positive functionals on $A$. Then there is a representable positive functional $f_r:A\to\mathbb{C}$ such that $f_r$ is the greatest among all of the representable positive functionals $f_0$ such that $f_0\leq f$ and $f_0$ is absolutely continuous with respect to $g$. The positive functional $f_s:=f-f_r$ is representable, moreover $f_s$ and $g$ are singular functionals, as well as $f_s$ and $f_r$.
	In notation:
	\[
	f=f_r+f_s;\ f_r\ll g,\ f_s\perp g,\ f_s\perp f_r.
	\]
	The sum $f=f_r+f_s$ is \textbf{the Lebesgue decomposition of $f$ with respect to $g$}, where $f_r$ is called the \textbf{absolutely continuous} or \textbf{regular part} of $f$ with respect to $g$, and $f_s$ is the \textbf{singular part} of $f$. 
\end{theorem}

\begin{remark}\label{why}
In \cite{SzLebalg}, Corollary 3.2, the statements of the previous theorem were immediate consequences of a result on representable forms over non-involutive complex algebras (\cite{SzLebalg}, Theorem 2.6). 
This was obtained directly from the form decomposition \eqref{formafelb}, which is originally due to Gheondea and Kavruk (\cite{GK}, Theorem 7.8). 
For the latter, a different proof can be found in \cite{HSdS} (Theorems 2.11 and 3.9).
However, since every positive functional on a $C^*$-algebra is a completely positive map, in the case of $C^*$-algebras, Theorem \ref{Lebesgue} follows from Theorem 3.1 in \cite{GK}, that is, the Lebesgue decomposition of completely positive maps. 
Moreover, the general version above can be obtained from the $C^*$-algebra version by the factorizing methods what we use in subsection \ref{nemGNS}.
Thus, since Theorem 3.1 in \cite{GK} lie upon Ando's theory (\cite{Ando}; \cite{GK}, Appendix A), Theorem \ref{Lebesgue} is a consequence of the Lebesgue decomposition theory of positive operators.

We must note that the similar decompositions of Gudder (\cite{G}, Corollary 3) and Kosaki (\cite{Ko}, Theorem 3.5) coincide with the decomposition in Theorem \ref{Lebesgue} in the same settings, namely, over unital Banach $^*$-algebras and von Neumann algebras (but in \cite{G} the maximality of the absolutely continuous part has not been pointed out).
\end{remark}

One major question remained after the existence Theorem \ref{Lebesgue}: what can we say about the uniqueness of the decomposition?
It is not an obvious one, as it can be seen from the results below. 
The main subject of this paper is to give a complete answer by the aid of the representation theory of $^*$-algebras.

We formulate the precise meaning of the uniqueness  in the next definition. 

\begin{definition}\label{defunique}
	We say that the Lebesgue decomposition $f=f_r+f_s$ of $f$ with respect to $g$ is \emph{unique}, if for arbitrary representable positive functionals $f_r'$ and $f_s'$ on $A$ the properties
	\[
	f=f_r'+f_s';\ f_r'\ll g,\ f_s'\perp g
	\]
	force that $f_r'=f_r$ and $f_s'=f_s$.
	
	Throughout the paper the phrase "\emph{the Lebesgue decomposition of representable positive functionals over $A$ is unique}" (or shortly, "\emph{the Lebesgue decomposition over $A$ is unique}") means that for every pair of representable positive functionals $f,g:A\to\mathbb{C}$ the Lebesgue decomposition of $f$ with respect to $g$ is unique.
\end{definition}

The general problem on the uniqueness is the following: what kind of assumptions do we have to make on the $^*$-algebra in order to imply the uniqueness of the Lebesgue decomposition of representable positive functionals over $A$?
In the Introduction, we mentioned several results on the uniqueness of the decomposition in the case of positive operators, forms and completely positive maps.
So, let us summarize here what progress have been made on this question over the years in the case of positive functionals. (See also in \cite{SzAbs}, section $5$).

The first (highly non-trivial) non-uniqueness result is due to Kosaki (\cite{Ko}, 10.5 and 10.6).
He gave an example of a properly infinite von Neumann algebra $M$ which has the property that the Lebesgue decomposition is not unique over $M$, even for normal positive functionals. (Note that over finite von Neumann algebras the uniqueness is true for normal positive functionals. We prove this in Corollary \ref{normaluniq}.)
In a recent paper (\cite{Zsig1}, Examples 6.4 and 6.6),  by the aid of Ando's non-uniqueness result on positive operators (\cite{Ando}, Theorem 6), Tarcsay and Titkos showed that the Lebesgue decomposition is not unique over the $^*$-algebra of compact operators on an infinite dimensional Hilbert space $\mathcal{H}$, as well as over the full operator algebra $\mathscr{B}(\mathcal{H})$. 
Despite the beautiful connection with Ando's decomposition for positive operators (\cite{Ando}), which was proved in \cite{Zsig1}, this way seems to be a dead end. The usability of Ando's non-uniqueness theorem lies upon the well-known dualities between the compact operators, the trace-class operators and the full operator algebra (\cite{Sa}, 1.19). 
Hence, in general (for instance, in the case of \emph{NGCR/antiliminal} $C^*$-algebras; \cite{Bla2}, Definition IV.1.3.1), Ando's theorem presumably cannot be applied.

For positive results on the uniqueness, we cite the paper \cite{SzAbs} of the first author. 
In section $5$, it was mentioned that for finite dimensional vector spaces the Lebesgue decomposition of forms is unique (in particular, for representable positive functionals defined on finite dimensional $^*$-algebras). 
It was showed that over commutative $^*$-algebras the Lebesgue decomposition of representable positive functionals is unique. 
The proof uses the Gelfand-Naimark theorem and the Riesz-representation theorem, hence, works only in the commutative case.
In \cite{SzAbs}, the uniqueness was also proved in the case of the Banach $^*$-algebras $L^1(G)$ where $G$ is a compact group. 

The results above show that the uniqueness is a rather complicated problem. Neither the commutativity, nor the finite dimensionality characterize it.   
It was already highlighted in \cite{SzAbs}, page 244 that the topologically irreducible representations of finite dimensional $^*$-algebras, commutative $^*$-algebras and  $L^1$-algebras of compact groups are all finite dimensional. 
The $^*$-algebras in the non-uniqueness examples have infinite dimensional irreducible representations. 
In the light of these observations, our main goal is to prove:

\begin{theorem}\label{main}
	If $A$ is a $^*$-algebra, then the Lebesgue decomposition of representable positive functionals over $A$ is unique if and only if every topologically irreducible representation of $A$ is finite dimensional.
\end{theorem}

Section \ref{biz}, the main part of the article, provides the proof of this theorem.
Section \ref{utols} collects examples and remarks related to the Lebesgue decomposition and the uniqueness over general $^*$-algebras.
In particular, we present examples which are not $C^*$-algebras, but admit many non-zero representable positive functionals and topologically irreducible representations. 
For these $^*$-algebras, the Lebesgue decomposition theory of representable positive functionals is non-tri\-vi\-al. 
To mention here such an example, the $L^1$-space of a Hausdorff locally compact group is a Banach $^*$-algebra in a natural manner. 
The finite dimensional property for irreducible representations in our Theorem \ref{main} indicates a connection with an important class of locally compact groups. The so-called Moore groups are the ones that have only finite dimensional continuous, unitary, topologically irreducible representations. This suggests that the uniqueness of the Lebesgue decomposition of positive functionals over the $L^1$-algebras actually characterizes Moore groups. We verify this in Corollary \ref{Moore}, and other examples from the classes of $G^*$-algebras and $C^*$-convex algebras are also included in Section \ref{utols}.

\section{PROOF OF THE THEOREM}\label{biz}

We divide the proof into three parts. 
First we show that if a $C^*$-algebra $A$ admits an infinite dimensional irreducible representation, then the Lebesgue decomposition of positive linear functionals over $A$ is not unique (Theorem \ref{infinite}). 
The argument makes a heavy use of Kadison's theorem on transitivity of irreducible representations of $C^*$-algebras (see Theorem \ref{Kadison} below), and a technical result on the uniqueness (Lemma \ref{uniqlem1}).  

The second part deals with the converse: if a $C^*$-algebra $A$ has only finite dimensional irreducible representations, then the Lebesgue decomposition over $A$ is unique (Theorem \ref{mainC}). The proof lies upon the following results: a theorem on extensions from closed hereditary $^*$-subalgebras (Corollary \ref{heredit}), the correspondence between the positive functionals on $A$ and the normal positive functionals on the enveloping von Neumann algebra $W^*(A)$ (Lemma \ref{abszlem}, Corollary \ref{fedoWun}), H. Kosaki's theorem on finite von Neumann algebras (Theorem \ref{Kosaki}) and M. Hamana's characterization of $C^*$-algebras with only finite dimensional irreducible representations (Theorem \ref{Hamana}). (We recall the property "hereditary" and the concept of the enveloping von Neumann algebra in the concrete subsection.)

From these two special cases, we conclude the characterization for general $^*$-algebras in the third part (Theorem \ref{mainnew}).

\subsection{$C^*$-algebras with infinite dimensional irreducible rep\-re\-sen\-ta\-ti\-ons}\label{ccsillagelo}
Before we prove our result on the non-uniqueness, we must record some facts of the representation theory of $C^*$-algebras.

R. V. Kadison's remarkable theorem states that a topologically irreducible representation of a $C^*$-algebra is algebraically irreducible (\cite{Kad}; Theorem 9.6.2 in \cite{Palmer2}). 
Hence, in the case of $C^*$-algebras, we simply speak about \emph{irreducible representations}.  

Every positive functional $f$ on a $C^*$-algebra $A$ is continuous and representable (\cite{Palmer2}, Theorem 9.5.17). 
If $A$ is unital, then for the identity $\mathbf{1}$ we have $f(\mathbf{1})=\Vert f\Vert$ (\cite{Dixmier}, Proposition 2.1.4). 
We say that a linear functional $f:A\to\mathbb{C}$ is a \emph{state}, if it is positive and $\Vert f\Vert=1$.  
A non-zero positive functional $p$ on $A$ is \emph{pure}, if for every $p':A\to\mathbb{C}$  positive functional, the inequality $p'\leq p$ implies the existence of a non-negative number $\lambda$ such that $p'=\lambda p$. By 2.5.4 in \cite{Dixmier}, this is equivalent to the property that the representation $\pi_p$ is irreducible, where $\pi_p$ is the representation associated with $p$ by the GNS-construction (see II.6.4 in \cite{Bla2}).
It is obvious that a positive multiple of a pure functional is pure.

The following version of Kadison's transitivity theorem on irreducible representations is extremely useful for us (see Lemma 5.4.2 and the proof of Theorem 5.4.3 in \cite{KR1}).
\begin{theorem}\label{Kadison}
	Let $\pi:A\to\mathscr{B}(\mathcal{H})$ be an irreducible representation of the $C^*$-algebra $A$ on the Hilbert space $\mathcal{H}$, and let $m\in\mathbb{N}^+$, $r\geq 0$.
	If $(\xi_k)_{1\leq k\leq m}$ and $(\eta_k)_{1\leq k\leq m}$ are vector systems in $\mathcal{H}$ such that
	$(\xi_k)_{1\leq k\leq m}$ is orthonormal and $\Vert\eta_k\Vert\leq r$ for every $1\leq k\leq m$, then there exists an element $a\in A$ with the properties  
	\[
	\pi(a)\xi_k=\eta_k\ (1\leq k\leq m),\ \Vert\pi(a)\Vert\leq r\sqrt{2m}.
	\]
\end{theorem}

We need the following statement (e.g.: \cite{SzAbs}, Lemma 5.2), which characterizes the decomposition's uniqueness by means of absolute continuity and singularity. It is also indispensable in the finite dimensional case, as well as in the proof of our main result, Theorem \ref{mainnew}.

\begin{lemma}\label{uniqlem1}
	Assume that $A$ is a $^*$-algebra. Then the following statements are equivalent.
	\begin{itemize}
		\item[(i)] The Lebesgue decomposition of representable positive functionals over $A$ is u\-ni\-que. 
		
		\item[(ii)] For all representable positive functionals $t, f$ and $g$ on $A$, the property
		\[
		t\leq f\ll g
		\]
		implies that $t\ll g$.
		
		\item[(iii)]\label{szing} For all representable positive functionals $f,g$ and $p$ on $A$, the properties 
		\[
		p\leq f\ll g;\ p\perp g
		\]
		imply that $p=0$.
	\end{itemize}  
\end{lemma}

\begin{remark}\label{nemuniqrem}
	We note here that if $p=0$ in (iii) is not fulfilled, then the Lebesgue decomposition of the representable functional $h:=f+p$ with respect to $g$ is not unique.
	Indeed, if $p\leq f\ll g,\ p\perp g$ with $p\neq0$ and the decomposition (Theorem \ref{Lebesgue}) would be unique for every pair of representable positive functionals, then for the functional $h$ the assumptions imply that $h_r=f$ and $h_s=p$ in the decomposition with respect to $g$. In accordance to the theorem, the regular and singular parts are singular to each other, i.e., $f\perp p$. But together with the inequality $p\leq f$, this is impossible by Remark \ref{altul}, if $p\neq0$.
	(See also the proof of Lemma 5.2 in \cite{SzAbs}).
\end{remark}

Our first result is the non-uniqueness over $C^*$-algebras with infinite dimensional irreducible representations.

\begin{theorem}\label{infinite}
	Let $A$ be a $C^*$-algebra and assume that $\pi:A\to\mathscr{B}(\mathcal{H})$ is an irreducible representation on the infinite dimensional Hilbert space $\mathcal{H}$. Then the Lebesgue decomposition of positive functionals over $A$ is not unique.
\end{theorem}
\begin{proof}
	We show that there exist non-zero positive functionals $f,g$ and $p$ on $A$ such that 
	\begin{equation}\label{nemuniq}
	p\leq f\ll g\ \mathrm{and}\ p\perp g.
	\end{equation}
	Thus, by Lemma \ref{uniqlem1} (iii), we conclude that the Lebesgue decomposition of positive functionals over $A$ is not unique.
	
	Let us choose an orthonormal system $(\xi_k)_{k\in\mathbb{N}^+}$ in $\mathcal{H}$ and put 
\[
\xi:=\sum_{k=1}^{+\infty}\frac{1}{2^k}\xi_k.
\] 
	Define the mappings $f,g,p:A\to\mathbb{C}$ by the next formulas:
	\[
	f(a):=\sum_{k=1}^{+\infty}\frac{1}{2^k}(\pi(a)\xi_k\!\mid\!\xi_k);\ g(a):=\sum_{k=1}^{+\infty}\frac{1}{10^k}(\pi(a)\xi_k\!\mid\!\xi_k);\ p(a):=(\pi(a)\xi \!\mid\!\xi).
	\]
	Since $\pi$ is a representation, it is easy to see that these functions are positive linear functionals on $A$, and for any $a\in A$
	\begin{equation}\label{fandg}
	f(a^*a)=\sum_{k=1}^{+\infty}\frac{1}{2^k}\Vert\pi(a)\xi_k\Vert^2;\ g(a^*a)=\sum_{k=1}^{+\infty}\frac{1}{10^k}\Vert\pi(a)\xi_k\Vert^2;\ p(a^*a)=\Vert\pi(a)\xi\Vert^2.
	\end{equation}
	Each of these functionals are not zero by the irreducibility of $\pi$, since every non-zero vector in $\mathcal{H}$ is cyclic, in particular, $\xi$ and $\xi_k$ are cyclic vectors for every $k\in\mathbb{N}^+$.
	Moreover, the positive functional $p$ is pure, since the vector states determined by irreducible representations are pure (e.g., \cite{KR2}, page 728). According to $\Vert\xi\Vert^2=\frac13$, the functional $3p$ is a pure vector state. 
	
	The inequality $p\leq f$ holds. Indeed, if $a\in A$, then by \eqref{fandg}, we have
	\[
	p(a^*a)=\Vert\pi(a)\xi\Vert^2=\left\Vert\pi(a)\left(\sum_{k=1}^{+\infty}\frac{1}{2^k}\xi_k\right)\right\Vert^2=\left\Vert\sum_{k=1}^{+\infty}\frac{1}{2^k}(\pi(a)\xi_k)\right\Vert^2\leq
	\]
	\[
	\left(\sum_{k=1}^{+\infty}\frac{1}{2^k}\Vert\pi(a)\xi_k\Vert\right)^2=\left(\sum_{k=1}^{+\infty}\sqrt{\frac{1}{2^k}}\cdot\frac{\Vert\pi(a)\xi_k\Vert}{\sqrt{2^k}}\right)^2\leq\left(\sum_{k=1}^{+\infty}\frac{1}{2^k}\right)\sum_{k=1}^{+\infty}\frac{\Vert\pi(a)\xi_k\Vert^2}{2^k}=
\]
\[
f(a^*a),
	\]
	where we applied the Cauchy-Schwarz inequality to the sequences $\left(\sqrt{\frac{1}{2^k}}\right)_{k\in\mathbb{N}^+}$ and $\left(\frac{\Vert\pi(a)\xi_k\Vert}{\sqrt{2^k}}\right)_{k\in\mathbb{N}^+}$ in the last estimate.

	The functional $f$ is absolutely continuous with respect to $g$. 
For all $n\in\mathbb{N}^+$ and $a\in A$, let
	\[
	f_n(a)=\sum_{k=1}^{n}\frac{1}{2^k}(\pi(a)\xi_k\!\mid\!\xi_k).
	\]
	Then for abritrary $a\in A$ and $n\in\mathbb{N}^+$, it is obvious that the inequality 
\[
f_n(a^*a)\leq f_{n+1}(a^*a)
\]
and the equation 
	\[
	f(a^*a)=\sup_{n\in\mathbb{N}}f_n(a^*a)
	\]
	hold true. 
Moreover,
	\[
	\frac{1}{10^n}f_n(a^*a)=\frac1{10^n}\sum_{k=1}^{n}\frac{1}{2^k}(\pi(a^*a)\xi_k\!\mid\!\xi_k)\leq \frac1{10^n}\sum_{k=1}^{n}(\pi(a^*a)\xi_k\!\mid\!\xi_k)\leq
	\]
	\[
	\leq\sum_{k=1}^{n}\frac{1}{10^k}(\pi(a^*a)\xi_k\!\mid\!\xi_k)\leq g(a^*a),
	\] 
	that is, $f_n\leq 10^ng$ for every $n\in\mathbb{N}^+$. Thus, Definition \ref{absdef} implies that $f\ll g$.

	To see that $p$ and $g$ are singular, we verify Definition \ref{singdef}. Suppose that $p'$ is a positive functional such that
	\[
	p'\leq p,\ p'\leq g.
	\]
	We have to prove that $p'=0$.
	First note that the pureness of $p$ implies the existence of a number $\lambda\geq0$ for which $p'=\lambda p$ is true, i.e.,
	$\lambda p\leq g$ holds.

	For every positive natural number $n$ let
	\[
	\xi_{n}':=\xi-\sum_{k=1}^n\frac{1}{2^k}\xi_k=\sum_{k=n+1}^{+\infty}\frac{1}{2^k}\xi_k.
	\]
	An easy calculation shows that
	\[
	\Vert\xi_n'\Vert=\left\Vert\sum_{k=n+1}^{+\infty}\frac{1}{2^k}\xi_k\right\Vert=\frac{1}{2^n\sqrt{3}}.
	\]
	Note that $(\xi_1,\dots,\xi_n,2^n\sqrt{3}\xi_{n}')$ is an orthonormal system in $\mathcal{H}$. 
Then, considering the system $(\eta_k)_{1\leq k\leq n+1}$ with $\eta_k=0$ for $1\leq k\leq n$ and $\eta_{n+1}=2^n\sqrt{3}\xi$, in view of the irreducibility of $\pi$, by the Kadison's Theorem \ref{Kadison}, there exists an element $a_n\in A$ for all $n\in\mathbb{N}^+$ such that
	\[
	\pi(a_n)\xi_k=0\ (1\leq k\leq n);\ \pi(a_n)[2^n\sqrt{3}\xi_n']=2^n\sqrt{3}\xi;\ \Vert\pi(a_n)\Vert\leq 2^n\sqrt{2(n+1)}.
	\]
	For the operator $\pi(a_n)$, we get
	\[
	\pi(a_n)\xi=\pi(a_n)\left(\sum_{k=1}^{+\infty}\frac{1}{2^k}\xi_k\right)=\pi(a_n)\left(\sum_{k=n+1}^{+\infty}\frac{1}{2^k}\xi_k\right)=\pi(a_n)\xi_n'=\xi.
	\]
	Hence, for the functionals $p$ and $g$, we obtain by \eqref{fandg} that    
	\[
	p(a_n^*a_n)=\Vert\pi(a_n)\xi\Vert^2=\Vert\xi\Vert^2=\frac13;
	\]
	\[
	g(a_n^*a_n)=\sum_{k=1}^{+\infty}\frac{1}{10^k}\Vert\pi(a_n)\xi_k\Vert^2=\sum_{k=n+1}^{+\infty}\frac{1}{10^k}\Vert\pi(a_n)\xi_k\Vert^2\leq
	\]
	\[
	\sum_{k=n+1}^{+\infty}\frac{1}{10^k}\Vert\pi(a_n)\Vert^2\Vert\xi_k\Vert^2
	\leq\left(2^n\sqrt{2(n+1)}\right)^2\sum_{k=n+1}^{+\infty}\frac{1}{10^k}=\frac{2}{9}\left(\frac{2}{5}\right)^n(n+1).
	\]
	Using the inequality $\lambda p=p'\leq g$, we conclude for every $n\in\mathbb{N}^+$ that
	\[
	\lambda=\lambda 3p(a_n^*a_n)\leq 3g(a_n^*a_n)\leq\frac{2}{3}\left(\frac{2}{5}\right)^n(n+1),
	\]
	so the number $\lambda$ is zero.
	We proved all of the conditions in \eqref{nemuniq},
	the non-u\-ni\-qu\-e\-ness follows.
\end{proof}

According to Remark \ref{nemuniqrem}, the Lebesgue decomposition of $f+p$ with respect to $g$ is not unique. Hence, our proof  gives concrete positive functionals for which the Lebesgue decomposition is not unique. 

From the arguments above, it is easy to obtain an example which shows the non-uniqueness of the Lebesgue decomposition of normal positive functionals on von Neumann algebras. 
Let $\mathcal{H}$ be an infinite dimensional Hilbert space. 
The identity representation of $\mathscr{B}(\mathcal{H})$ on $\mathcal{H}$ is irreducible, thus, the arguments above obviously works well for $A:=\mathscr{B}(\mathcal{H})$ and the identity representation. Furthermore, the functionals $f, g$ and $p$ are normal functionals on the von Neumann algebra $\mathscr{B}(\mathcal{H})$ (\cite{Bla2}, III.2.1.4).
(We note that this $^*$-algebra also appeared in Example 6.6 of \cite{Zsig1} to reveal the non-uniqueness over von Neumann algebras.
However, the arguments make a heavy use of Ando's non-trivial theorems from \cite{Ando}.  
In contrast to this example, we have no need for the results of Ando.)

\subsection{$C^*$-algebras having only finite dimensional irreducible representations}
In this part, we prove the uniqueness over $C^*$-algebras without infinite dimensional irreducible representations (Theorem \ref{mainC}).
The main idea is that the uniqueness problem can be transferred to the case of enveloping von Neumann algebras and normal functionals. That is, the decomposition is unique over a $C^*$-algebra $A$ precisely when the decomposition is unique among the normal positive functionals on the enveloping von Neumann algebra $W^*(A)$ (Corollary \ref{fedoWun}). By the aid of highly non-trivial results of Kosaki (Theorem \ref{Kosaki}) and Hamana (Theorem \ref{Hamana}), we are able to prove the uniqueness over the latter $^*$-algebra.

For this purpose, we need some preparations. First, we obtain results on positive extensions from closed hereditary $^*$-subalgebras which preserves absolute continuity. After that, we turn to enveloping von Neumann algebras and normal positive functionals (see the discussion after Remark \ref{refher}).

We recall some facts about positive extensions of functionals defined on a $C^*$-subalgebra $B$. 
(By an \emph{extension} of a positive functional we always mean a positive linear extension.)
It is obvious that the norm of every extension is equal or greater than the norm of the functional on $B$.
Moreover, every $h:B\to\mathbb{C}$ positive functional has a norm preserving extension to $A$ (\cite{Pedersen}, Proposition 3.1.6).

If $h_1,h_2:B\to\mathbb{C}$ are positive functionals, $\widetilde{h_1}$ and $\widetilde{h_2}$ are norm preserving extensions to $A$, then their sum $\widetilde{h_1}+\widetilde{h_2}$ is a norm preserving extension of $h_1+h_2$. 
Indeed, it clearly extends $h_1+h_2$. Furthermore, using the terminology of 1.4 in \cite{Pedersen}, for every approximate unit $(e_i)_{i\in I}$ in $B$,  the net $(f(e_i))_{i\in I}$ converges to $\Vert f\Vert$ for a positive functional $f$ defined on $B$ (\cite{Pedersen}, 3.1.4). Thus,
on the one hand,
\begin{equation}\label{normsum1}
\Vert h_1\Vert+\Vert h_2 \Vert=\lim_{i,I}h_1(e_i)+\lim_{i,I}h_2(e_i)=\lim_{i,I}(h_1+h_2)(e_i)=\Vert h_1+h_2\Vert.
\end{equation}
On the other hand, if the net $(\widetilde{e_j})_{j\in J}$ is an approximate unit in $A$, then norm preservity implies
\begin{equation}\label{normsum2}
\Vert h_1\Vert+\Vert h_2 \Vert=\lim_{j,J}\widetilde{h_1}(\widetilde{e_j})+\lim_{j,J}\widetilde{h_2}(\widetilde{e_j})=\lim_{j,J}(\widetilde{h_1}+\widetilde{h_2})(\widetilde{e_j})=\Vert \widetilde{h_1}+\widetilde{h_2}\Vert.
\end{equation}
Hence, by \eqref{normsum1} and \eqref{normsum2}, we conclude that 
\begin{equation}\label{normpres1}
\Vert \widetilde{h_1}+\widetilde{h_2}\Vert=\Vert\widetilde{h_1}\Vert+\Vert\widetilde{h_2}\Vert=\Vert h_1\Vert+\Vert h_2\Vert=\Vert h_1+h_2\Vert.
\end{equation}

The following lemma shows the existence of norm preserving extensions which also preserve the relation $\ll$.

\begin{lemma}\label{closedsubext}
	If $A$ is a $C^*$-algebra, then for positive functionals $f$ and $g$ defined on a closed $^*$-subalgebra $B$ of $A$ the following are true:
	\begin{enumerate}
		\item[(1)] $f\leq g$ $\Leftrightarrow$ there exist norm preserving extensions $\widetilde{f}$, $\widetilde{g}$ of $f$ and $g$ to $A$ such that $\widetilde{f}\leq\widetilde{g}$.
		
		\item[(2)] $f\ll g$ $\Leftrightarrow$ there exist norm preserving extensions $\widetilde{f}$, $\widetilde{g}$ of $f$ and $g$ to $A$ such that $\widetilde{f}\ll\widetilde{g}$.
		
		%\item $f\perp g$ if and only if $\widetilde{f}\perp\widetilde{g}$ holds for every norm preserving extensions $\widetilde{f}$, $\widetilde{g}$ of $f$ and $g$ to $A$. 
	\end{enumerate}
\end{lemma}
\begin{proof}
	$\Leftarrow$:  If there are norm preserving extensions $\widetilde{f}$ and $\widetilde{g}$ such that $\widetilde{f}\leq\widetilde{g}$ or $\widetilde{f}\ll\widetilde{g}$, then from the definitions it is obvious that these relations hold for the restrictions $\widetilde{f}\!\mid\!_B=f$ and $\widetilde{g}\!\mid\!_B=g$. 
	
	$\Rightarrow$: \emph{(1)}:
	Suppose that $f\leq g$. This means that the functional $g-f$ is positive, $g=(g-f)+f$, so for arbitrary norm preserving extensions $\widetilde{f}$ and $\widetilde{g-f}$ to $A$, we have that the functional
	\[
	\widetilde{g}:=\widetilde{g-f}+\widetilde{f}
	\]
	is a positive extension of $g$. 
	Trivially $\widetilde{f}\leq\widetilde{g}$ holds, moreover $\Vert g\Vert=\Vert\widetilde{g}\Vert$, thanks to \eqref{normpres1}:
	\[
	\Vert\widetilde{g}\Vert=\Vert\widetilde{g-f}+\widetilde{f}\Vert=\Vert\widetilde{g-f}\Vert+\Vert\widetilde{f}\Vert=\Vert g-f\Vert+\Vert f\Vert=\Vert (g-f)+f\Vert=\Vert g\Vert.
	\]
	
	$\Rightarrow$: \emph{(2)}: Assume that $f\ll g$. By Definition \ref{absdef}, fix a sequence of positive numbers
	$(\alpha_n)_{n\in\mathbb{N}}$ and an increasing sequence $(f_n)_{n\in\mathbb{N}}$ of positive functionals on $B$ with the conditions  
	\[
	\sup_{n\in\mathbb{N}}f_n=f;\ f_n\leq\alpha_ng\ (n\in\mathbb{N}). 
	\]
	We show by induction that there is an increasing sequence $(\widetilde{f_n})_{n\in\mathbb{N}}$ of positive functionals on $A$ such that for any $n\in\mathbb{N}$ the functional $\widetilde{f_n}$ is a norm preserving extension of $f_n$. To see this, let
	$\widetilde{f_0}$ be an arbitrary norm preserving extension of $f_0$ to $A$, and let $\widetilde{f_{n+1}-f_n}$ be an arbitrary norm preserving extension of $f_{n+1}-f_n$ for any $n\in\mathbb{N}$. Now define for every natural number $n$ the positive functional $\widetilde{f_{n+1}}$ by the recursion
	\begin{equation}\label{normpres2}
	\widetilde{f_{n+1}}:=\widetilde{f_{n+1}-f_n}+\widetilde{f_n}.
	\end{equation}
	From this equation, it is obvious that $\widetilde{f_n}\leq\widetilde{f_{n+1}}$ is true for all $n$. By induction, it is also clear that these functionals are extensions: $\widetilde{f_0}(b)=f_0(b)$ for any $b\in B$, furthermore
	\[
	\widetilde{f_{n+1}}(b)=\widetilde{f_{n+1}-f_n}(b)+\widetilde{f_n}(b)=(f_{n+1}-f_n)(b)+f_n(b)=f_{n+1}(b).
	\]
	The equalities for the norms also can be proved with induction: for $n=0$ we have chosen $\widetilde{f_0}$ to be norm preserving. Now, if we suppose that $\Vert\widetilde{f_n}\Vert=\Vert f_n\Vert$ for an $n\in\mathbb{N}$, then by \eqref{normpres1} and \eqref{normpres2}, 
	\[ 
	\Vert\widetilde{f_{n+1}}\Vert=\Vert\widetilde{f_{n+1}-f_n}\Vert+\Vert\widetilde{f_n}\Vert=\Vert f_{n+1}-f_n\Vert+\Vert f_n\Vert=\Vert f_{n+1}\Vert.
	\]
	Since for any $a\in A$
	\[
	\lvert\widetilde{f_n}(a)\rvert\leq\Vert f_n\Vert\Vert a\Vert\leq\Vert f\Vert \Vert a\Vert
	\]
	is true, the supremum functional $\widetilde{f}:=\sup_{n\in\mathbb{N}}\widetilde{f_n}$ exists. 
The previous estimate also shows that $\Vert \widetilde{f}\Vert\leq\Vert f\Vert$. But $\widetilde{f}$ is an extension of $f$ (which is then norm preserving from the inequality), because for all $b\in B$ 
	\[
	f(b)=\lim_{n\to+\infty}f_n(b)=\lim_{n\to+\infty}\widetilde{f_n}(b)=\widetilde{f}(b).
	\] 
	
	Now we show the existence of a norm preserving extension $\widetilde{g}$ of $g$ and a sequence $(\widetilde{\alpha_n})_{n\in\mathbb{N}}$ of positive numbers with the property
	\begin{equation}\label{egyen}
	\widetilde{f_n}\leq\widetilde{\alpha_n}\widetilde{g}\ (n\in\mathbb{N}).
	\end{equation}
	Together with the properties of the increasing sequence $(\widetilde{f_n})_{n\in\mathbb{N}}$, this leads to the conclusion $\widetilde{f}\ll\widetilde{g}$.

	For every $n\in\mathbb{N}$, choose $\widetilde{\alpha_ng-f_n}$ to be an arbitrary norm preserving extension of the positive functional $\alpha_ng-f_n$.
	Let
	\[
	\widetilde{g_n}:=\frac{1}{\alpha_n}\left(\widetilde{\alpha_ng-f_n}+\widetilde{f_n}\right).
	\]
It is easy to see that, for all $n\in\mathbb{N}$, $\widetilde{g_n}$ is an extension of $g$: 
\[
\widetilde{g_n}(b)=\frac{1}{\alpha_n}\left(\widetilde{\alpha_ng-f_n}+\widetilde{f_n}\right)(b)=\frac{1}{\alpha_n}\left((\alpha_ng-f_n)(b)+f_n(b)\right)=g(b)\ \ \ (b\in B).
\] 
Furthermore, by \eqref{normpres1}, the equation $\Vert\widetilde{g_n}\Vert=\Vert g\Vert$ holds:
\[
\Vert\widetilde{g_n}\Vert=\frac{1}{\alpha_n}\left\Vert\widetilde{\alpha_ng-f_n}+\widetilde{f_n}\right\Vert=\frac{1}{\alpha_n}\left\Vert(\alpha_ng-f_n)+f_n \right\Vert=\Vert g\Vert
\]

	Take a sequence $(\delta_m)_{m\in\mathbb{N}}$ of positive numbers with sum $1$. Define the functional $\widetilde{g}$ by the equation
	\[
	\widetilde{g}:=\sum_{m=0}^{+\infty}\delta_m\widetilde{g_m}.
	\] 
	The sum converges in the functional-norm and this is a norm preserving extension of $g$. 
Indeed,
	\[
	\Vert\widetilde{g}\Vert\leq\sum_{m=0}^{+\infty}\delta_m\Vert \widetilde{g_m}\Vert=\sum_{m=0}^{+\infty}\delta_m\Vert g\Vert=\Vert g\Vert,
	\]
	moreover for every $b\in B$ we infer that  
	\[
	\widetilde{g}(b)=\sum_{m=0}^{+\infty}\delta_m\widetilde{g_m}(b)=\sum_{m=0}^{+\infty}\delta_mg(b)=g(b).
	\]
	Lastly, for $n\in\mathbb{N}$, let 
	\[
	\widetilde{\alpha_n}:=\frac{\alpha_n}{\delta_n}.
	\]
	With these numbers
	\[
	\widetilde{\alpha_n}\widetilde{g}=\frac{\alpha_n}{\delta_n}\sum_{m=0}^{+\infty}\delta_m\widetilde{g_m}\geq \frac{\alpha_n}{\delta_n}\delta_n\widetilde{g_n}=\alpha_n\frac{1}{\alpha_n}\left(\widetilde{\alpha_ng-f_n}+\widetilde{f_n}\right)\geq\widetilde{f_n},
	\]
	that is, \eqref{egyen} holds. The proof is complete.
\end{proof}

Let $A$ be a $C^*$-algebra and denote by $A_+$ the set of the positive elements in $A$.
A $C^*$-subalgebra $B$ of $A$ is \emph{hereditary} (\cite{Pedersen}, 1.5; \cite{Bla2}, II.3.4), if whenever $a\in A_+$ and $b\in B_+$ are such that $a\leq b$, it follows that $a\in B$. The standard examples of hereditary $C^*$-subalgebras are the closed ideals of $A$ and the \emph{corner $^*$-subalgebras} $eAe$ with a projection $e\in A$.

In Proposition 3.1.6 of \cite{Pedersen}, it was shown that every positive functional defined on a closed hereditary $^*$-subalgebra has a \emph{unique} norm preserving extension to $A$. In fact, this property characterizes hereditary $C^*$-subalgebras (this is a theorem of M. Kusuda (\cite{Ku})).

The following result plays an important role in the proof of Lemma \ref{finitelemma}.

\begin{corollary}\label{heredit}
	Let $B$ be a closed hereditary $^*$-subalgebra of the $C^*$-algebra $A$. Let $f,g:B\to\mathbb{C}$ be positive functionals, and denote by $\widetilde{f}$ (resp. $\widetilde{g}$) the unique positive norm preserving extension of $f$ (resp. $g$) to $A$.
	Then the following equivalences hold. 
	\begin{enumerate}
		\item[(1)] $f\leq g$ $\Leftrightarrow$ $\widetilde{f}\leq\widetilde{g}$.
		
		\item[(2)] $f\ll g$ $\Leftrightarrow$ $\widetilde{f}\ll\widetilde{g}$.
		
		%\item $f\perp g$ $\Leftrightarrow$ $\widetilde{f}\perp\widetilde{g}$.
	\end{enumerate}
\end{corollary}
\begin{proof}
	These statements are immediate from Lemma \ref{closedsubext}.
\end{proof}

\begin{remark}\label{refher}
	We note here that there are similar characterizations for singularity via the extensions. Namely, for positive functionals $f$ and $g$ defined on a closed $^*$-subalgebra $B$ of a $C^*$-algebra $A$, $f\perp g$ is equivalent to the relation $\widetilde{f}\perp\widetilde{g}$ for \emph{every} norm preserving extensions of $f$ and $g$ to $A$. If $B$ is hereditary, then the latter property reduces to the pair of the unique norm preserving extensions $\widetilde{f}$ and $\widetilde{g}$. 
	
	Moreover, in the case of hereditary $B$, one can show the following correspondence between the Lebesgue decompositions (Theorem \ref{Lebesgue}) of the original and the extended functionals:
	\[
	\widetilde{f_r}=(\widetilde{f})_r;\ \widetilde{f_s}=(\widetilde{f})_s.
	\]
	Neither of these statements are required henceforth, so we omit the proofs.     
\end{remark}

As we mentioned before, we need the notion of the enveloping von Neumann algebra of a $C^*$-algebra. 
Every property and idea we introduce here can be found in the standard textbooks (for example: \cite{Bla2}, \cite{Dixmier}, \cite{KR2}, \cite{Palmer2}, \cite{Pedersen}). 

Let $A$ be a $C^*$-algebra. Let $\Upsilon:A\to\mathscr{B}(\mathcal{H}_{\Upsilon})$ be the universal representation of $A$, that is, $\Upsilon$ is the orthogonal direct sum of all of the representations $\pi_f$ obtained by the GNS-construction (\cite{Bla2}, II.6.4; \cite{Palmer2}, section 9.4), where $f$ runs through the set of all positive functionals on $A$ (\cite{Bla2}, III.5.2; \cite{Palmer2}, Definition 10.1.5). 
It is known that $\Upsilon$ is a non-degenerate and faithful representation of $A$ on the Hilbert space $\mathcal{H}_{\Upsilon}$ (this is the orthogonal direct sum of the Hilbert spaces $\mathcal{H}_f$).
Let $W^*(A):=(\Upsilon\left<A\right>)''$, i.e., the bicommutant of $\Upsilon$. Then $W^*(A)$ is a von Neumann algebra, what we call \emph{the enveloping von Neumann algebra of $A$}.
(The notations $A''$ and $A^{**}$ are also common (\cite{Pedersen}, 3.7.6, \cite{Bla2}, III.5.2); the latter refers to the fact that the enveloping von Neumann algebra $W^*(A)$ and the second Banach dual of $A$ are isomorphic as Banach spaces.)
Following the usual terminology (indentifying $A$ with $\Upsilon\left<A\right>$), we regard $A$ as a norm closed $^*$-subalgebra of $W^*(A)$.

Every positive functional $f$ on $A$ can be uniquely extended to a normal positive functional $\mathfrak{f}$ on the enveloping von Neumann algebra $W^*(A)$ . 
The extension mapping $f\mapsto\mathfrak{f}$ is an isometric bijection onto the set of normal positive functionals, which is order preserving, that is,
\begin{equation}\label{fedoegyen}
f\leq g\ \Leftrightarrow\ \mathfrak{f}\leq \mathfrak{g}.
\end{equation}
See, for example, 12.1.3 in \cite{Dixmier}.

Now we can continue gathering the facts that support our proof of the uniqueness. 
The next stament shows that absolute continuity between positive functionals also holds true between the normal extensions: 

\begin{lemma}\label{abszlem}
	Let $A$ be a $C^*$-algebra and let $W^*(A)$ be its enveloping von Neumann algebra. With the preceding notations, we have the equivalence 
	\[
	f\ll g\ \Leftrightarrow\ \mathfrak{f}\ll \mathfrak{g}
	\]
	for positive functionals $f$ and $g$ on $A$.
	%\item $f\perp g$ $\Leftrightarrow$ $\mathfrak{f}\perp \mathfrak{g}$.
\end{lemma}
\begin{proof}
	Since absolute continuity passes down to restrictions (Remark \ref{altul}), $f\ll g$ follows from $\mathfrak{f}\ll \mathfrak{g}$.
	
	For the converse, assume that $f$ is absolutely continuous with respect to $g$. 
	Then, by Definition \ref{absdef}, there exist a sequence of positive numbers
	$(\alpha_n)_{n\in\mathbb{N}}$ and an increasing sequence $(f_n)_{n\in\mathbb{N}}$ of positive functionals on $A$ with the properties
	\[
	\sup_{n\in\mathbb{N}}f_n=f;\ f_n\leq\alpha_ng\ (n\in\mathbb{N}). 
	\]
	Therefore, the order preserving property \eqref{fedoegyen} implies for any $n\in\mathbb{N}$ that 
	\[
	\mathfrak{f_n}\leq\mathfrak{f_{n+1}}\leq\mathfrak{f};\ \mathfrak{f_n}\leq\alpha_n\mathfrak{g}.
	\]  
	Hence, for $\mathfrak{f}\ll \mathfrak{g}$, it is enough to show that $\sup_{n\in\mathbb{N}}\mathfrak{f_n}=\mathfrak{f}$.
	The previous inequalities imply that $\sup_{n\in\mathbb{N}}\mathfrak{f_n}\leq\mathfrak{f}$, so $\sup_{n\in\mathbb{N}}\mathfrak{f_n}$ is a normal positive functional on $W^*(A)$, since $\mathfrak{f}$ is normal. But, by the equality $\sup_{n\in\mathbb{N}}f_n=f$, it is an extension of $f$. Thus, the bijection between the positive functionals on $A$ and the normal extensions to $W^*(A)$ proves the equality $\sup_{n\in\mathbb{N}}\mathfrak{f_n}=\mathfrak{f}$.
\end{proof}

\begin{remark}\label{messe}
Similar to the case of hereditary subalgebras (Remark \ref{refher}), the unique normal extensions to $W^*(A)$ preserve singularity, that is, $f\perp g$ $\Leftrightarrow$ $\mathfrak{f}\perp\mathfrak{g}$. 
	Furthermore, the normal extension preserves the regular and singular parts in the Lebesgue decompositions (Theorem \ref{Lebesgue}) with respect to $g$ and $\mathfrak{g}$: 
	\[
	\mathfrak{f_r}=(\mathfrak{f})_r;\ \mathfrak{f_s}=(\mathfrak{f})_s.
	\]
	Since these properties are not relevant in this paper, we omit their proofs.      
\end{remark}

The following analogue of  Lemma \ref{uniqlem1} deals with the uniqueness of the Lebesgue decomposition of normal positive functionals on a von Neumann algebra. That is, we say for a von Neumann algebra $M$ that the Lebesgue decomposition of normal positive functionals over $M$ is \emph{unique}, if
for every pair of normal positive functionals $\mathfrak{f}$ and $\mathfrak{g}$ on $M$ the decomposition of $\mathfrak{f}$ with respect to $\mathfrak{g}$ is unique.
We note that the uniqueness for normal positive functionals does not imply the uniqueness for positive functionals over $M$ in general (see (1) in Question \ref{Q1}).  

\begin{lemma}\label{uniqlem2}
	Let $M$ be a von Neumann algebra.
	Then the following statements are equivalent.
	\begin{itemize}
		\item[(i)] The Lebesgue decomposition of normal positive functionals over $M$ is unique. 
		
		\item[(ii)] For all normal positive functionals $\mathfrak{t}, \mathfrak{f}$ and $\mathfrak{g}$ on $M$, the property
		\[
		\mathfrak{t}\leq \mathfrak{f}\ll \mathfrak{g}
		\]
		implies that $\mathfrak{t}\ll \mathfrak{g}$.
	\end{itemize}
\end{lemma}
\begin{proof}
	The proof is the same as in Lemma 5.2 of \cite{SzAbs} (just substitute the property "representable" with "normal"). 
\end{proof}

We have arrived to another important result:

\begin{corollary}\label{fedoWun}
	Let $A$ be a $C^*$-algebra with enveloping von Neumann algebra $W^*(A)$. The next statements are equivalent.
	\begin{itemize}
		\item[(i)] The Lebesgue decomposition of positive functionals over $A$ is unique.
		\item[(ii)] The Lebesgue decomposition of normal positive functionals over $W^*(A)$ is u\-ni\-qu\-e.
	\end{itemize}
\end{corollary}
\begin{proof}
	According to Lemma \ref{uniqlem1} (ii), the uniqueness over $A$ is equivalent to the hypothesis that
	for any positive functionals $t, f$ and $g$ on $A$, the condition
	\[
	t\leq f\ll g
	\]
	implies that $t\ll g$. By the connections between the positive functionals on $A$ and the normal positive functionals on $W^*(A)$ ( \eqref{fedoegyen} and Lemma \ref{abszlem}), this is equivalent to the property that for all normal positive functionals $\mathfrak{t}, \mathfrak{f}$ and $\mathfrak{g}$ on $W^*(A)$, the condition
	\[
	\mathfrak{t}\leq \mathfrak{f}\ll \mathfrak{g}
	\]
	implies that $\mathfrak{t}\ll \mathfrak{g}$.
	But Lemma \ref{uniqlem2} shows that this property holds exactly when the Lebesgue decomposition of normal positive functionals over $W^*(A)$ is u\-ni\-qu\-e.
\end{proof}

The uniqueness over $C^*$-algebras highly depends on the following Theorem \ref{Kosaki} of Kosaki on $\sigma$-finite von Neumann algebras (\cite{Ko}: Corollaries 2.3 and 2.4., there is no need to assume that $\mathfrak{g}$ is a state). Moreover, even the uniqueness for the decomposition of normal positive functionals over finite von Neumann algebras follows from this theorem. (Corollary \ref{normaluniq} below; this fact is also new, since it has not been pointed out in Kosaki's paper \cite{Ko}). 

We recall here that a von Neumann algebra $M$ is \emph{finite}, if for every $x\in M$, the equation $x^*x=\mathbf{1}$ implies $xx^*=\mathbf{1}$. It is well known that if $M$ is a finite von Neumann algebra, then for any projection $e$ in $M$ the corner von Neumann algebra $eMe$ is finite as well (\cite{Bla2}, \cite{DvN}, \cite{Sa}, \cite{Ta1}).

\begin{theorem}\label{Kosaki}
	For a von Neumann algebra $M$ with a faithful normal positive functional $\mathfrak{g}$ on $M$, the following are equivalent.
	\begin{itemize}
		\item[(i)] $M$ is finite.
		\item[(ii)] Every normal positive functional on $M$ is absolutely continuous with respect to $\mathfrak{g}$.
	\end{itemize}
\end{theorem}

For further properties, we have to record here some significant observations on the support projections and the norm preserving extensions of normal positive functionals (\cite{Bla2}, page 254;
\cite{DvN}, page 63;  \cite{Ta1}, page 134).

Let $M$ be a von Neumann algebra. If $\mathfrak{g}$ is a normal positive functional on $M$, then the left kernel 
\[
L_{\mathfrak{g}}=\{a\in M \!\mid\!\mathfrak{g}(a^*a)=0\}
\] 
is a left ideal generated by a projection $e_{\mathfrak{g}}\in M$, i. e., $L_{\mathfrak{g}}=Me_{\mathfrak{g}}$.
By definition, the \emph{support projection} of $\mathfrak{g}$ is $s(\mathfrak{g})=\mathbf{1}-e_{\mathfrak{g}}$. 
For every $x\in M$, we have
\begin{equation}\label{faith}
\mathfrak{g}(x)=\mathfrak{g}(s(\mathfrak{g})x)=\mathfrak{g}(xs(\mathfrak{g}))=\mathfrak{g}(s(\mathfrak{g})xs(\mathfrak{g})),
\end{equation}
moreover $\mathfrak{g}\!\mid\!_{s(\mathfrak{g})Ms(\mathfrak{g})}$ is a faithful normal positive functional on the corner von Neumann algebra $s(\mathfrak{g})Ms(\mathfrak{g})$.
In particular, since the unit element of the latter von Neumann algebra is $s(\mathfrak{g})$ and positive functionals on unital $C^*$-algebras attain their norm at the unit,
the equation \eqref{faith} for $x=\mathbf{1}$ gives
\[
\Vert \mathfrak{g}\Vert=\mathfrak{g}(\mathbf{1})=\mathfrak{g}(s(\mathfrak{g}))=\mathfrak{g}\!\mid\!_{s(\mathfrak{g})Ms(\mathfrak{g})}(s(\mathfrak{g}))=\Vert\mathfrak{g}\!\mid\!_{s(\mathfrak{g})Ms(\mathfrak{g})}\Vert.
\]
Together with the remarks about hereditary $^*$-subalgebras preceding to Corollary \ref{heredit}, this norm equality actually shows that the unique norm preserving positive extension of the positive functional $\mathfrak{g}\!\mid\!_{s(\mathfrak{g})Ms(\mathfrak{g})}$ from the norm closed hereditary $^*$-subalgebra $s(\mathfrak{g})Ms(\mathfrak{g})$ to $M$ is $\mathfrak{g}$.

Suppose that $\mathfrak{f}$ is another normal positive functional on $M$. 
Then, by Proposition 5 on page 65 in \cite{DvN}, the inclusion $L_{\mathfrak{g}}\subseteq L_{\mathfrak{f}}$ is equivalent to the inequality $s(\mathfrak{f})\leq s(\mathfrak{g})$ between the supports. Since $s(\mathfrak{g})\leq\mathbf{1}$, the positivity of $\mathfrak{f}$ implies that
\[
\mathfrak{f}(s(\mathfrak{f}))\leq\mathfrak{f}(s(\mathfrak{g}))\leq \mathfrak{f}(\mathbf{1})=\Vert\mathfrak{f}\Vert.
\]
But the argument presented in the previous paragraph shows that $\mathfrak{f}(s(\mathfrak{f}))=\mathfrak{f}(\mathbf{1})$. Hence, 
\[
\Vert\mathfrak{f}\Vert=\mathfrak{f}(s(\mathfrak{g}))=\Vert\mathfrak{f}\!\mid\!_{s(\mathfrak{g})Ms(\mathfrak{g})}\Vert,
\]
that is, $\mathfrak{f}$ is the unique norm preserving positive extension of $\mathfrak{f}\!\mid\!_{s(\mathfrak{g})Ms(\mathfrak{g})}$ to $M$.

The next statement is the key to the uniqueness over $C^*$-algebras. It is a consequence of the preceding remarks, Kosaki's theorem and the result dealing with extensions from hereditary subalgebras (Corollary \ref{heredit}). 

\begin{lemma}\label{finitelemma}
	Let $M$ be a finite von Neumann algebra and let $\mathfrak{f},\mathfrak{g}:M\to\mathbb{C}$ be normal positive functionals.
	Then the inclusion $L_{\mathfrak{g}}\subseteq L_{\mathfrak{f}}$ for the left kernels implies that $\mathfrak{f}$ is absolutely continuous with respect to $\mathfrak{g}$.
\end{lemma}
\begin{proof}
	The finiteness of $M$ implies that the von Neumann algebra $s(\mathfrak{g})Ms(\mathfrak{g})$ is finite too, which admits the faithful normal positive functional $\mathfrak{g}\!\mid\!_{s(\mathfrak{g})Ms(\mathfrak{g})}$. Hence, by Kosaki's Theorem \ref{Kosaki}, we conclude
	\[
	\mathfrak{f}\!\mid\!_{s(\mathfrak{g})Ms(\mathfrak{g})}\ll\mathfrak{g}\!\mid\!_{s(\mathfrak{g})Ms(\mathfrak{g})}.
	\]
	From the remarks above related to the unique norm preserving extensions of these functionals to $M$, Corollary \ref{heredit} \emph{(2)} guarantees that $\mathfrak{f}$ is absolutely continuous with respect to $\mathfrak{g}$.
\end{proof}

\begin{remark}\label{clos}
	In the context of Hilbert space operators, the preceding lem\-ma says the following. 
	For a positive functional $\mathfrak{f}$ on $M$, let $\mathcal{H}_{\mathfrak{f}}$ be the Hilbert space associated to $\mathfrak{f}$ by the GNS-construction, that is, $\mathcal{H}_{\mathfrak{f}}$ is the completion of the pre-Hilbert space $M/L_{\mathfrak{f}}$ endowed with inner the product 
	\[
	(\cdot\!\mid\!\cdot)_{\mathfrak{f}}:M/L_{\mathfrak{f}}\times M/L_{\mathfrak{f}}\to\mathbb{C};\ (a+L_{\mathfrak{f}}\!\mid\!b+L_{\mathfrak{f}})_{\mathfrak{f}}:=\mathfrak{f}(b^*a).
	\]
	Then the statement of the lemma is that if 
	\[
	\mathcal{H}_{\mathfrak{g}}\ni a+L_{\mathfrak{g}}\mapsto a+L_{\mathfrak{f}}\in\mathcal{H}_{\mathfrak{f}}    
	\]
	is a well-defined mapping, then it is automatically (a densely defined) closable operator for normal positive functionals $\mathfrak{f}$ and $\mathfrak{g}$ on a finite von Neumann algebra $M$.
\end{remark}

From the previous facts, we immediately get the uniqueness for the decomposition of normal positive functionals over finite von Neumann algebras. 
In Question \ref{Q1}, we examine the relationship of the normal-uniqueness and the $C^*$-algebra-uniqueness over von Neumann algebras closer.

\begin{corollary}\label{normaluniq}
	On finite von Neumann algebras the Lebesgue decomposition of normal positive functionals is unique.
\end{corollary}
\begin{proof}
	Let $M$ be a finite von Neumann algebra. We have to check property (ii) in Lemma \ref{uniqlem2}. 
	Take normal positive functionals $\mathfrak{t}, \mathfrak{f}$ and $\mathfrak{g}$ on $M$ with the hypothesis
	\[
	\mathfrak{t}\leq \mathfrak{f}\ll \mathfrak{g}.
	\]
	We need $\mathfrak{t}\ll \mathfrak{g}$.
	From the assumptions and \eqref{WDtul} in section \ref{prelims}, we get the following inclusions for the left kernels:
	\[
	L_{\mathfrak{g}}\subseteq L_{\mathfrak{f}}\subseteq L_{\mathfrak{t}}.
	\] 
	Thus, the previous Lemma \ref{finitelemma} implies $\mathfrak{t}\ll \mathfrak{g}$.
\end{proof}

We need the following characterization of finite enveloping von Neumann algebras from Hamana's paper (\cite{Ham}, Lemma 5):

\begin{theorem}\label{Hamana}
	Let $A$ be a $C^*$-algebra. 
	Then every irreducible representation of $A$ is finite dimensional if and only if $W^*(A)$ is a finite von Neumann algebra. 
\end{theorem}

We are in position to prove the $C^*$-algebra version of our main Theorem \ref{main}. 

\begin{theorem}\label{mainC}
	If $A$ is a $C^*$-algebra, then the following statements are equivalent.
	\begin{itemize}
		\item[(i)]The Lebesgue decomposition of positive functionals over $A$ is unique.
		\item[(ii)]Every irreducible representation of $A$ is finite dimensional.
	\end{itemize}
\end{theorem}
\begin{proof}
	(i)$\Rightarrow$(ii): This implication was proved in Theorem \ref{infinite}.
	
	(ii)$\Rightarrow$(i):
	If every irreducible representation of $A$ is finite dimensional, then by Hamana's Theorem \ref{Hamana}, the enveloping von Neumann algebra $W^*(A)$ of $A$ is finite. Hence, it follows from Corollary \ref{normaluniq} that the Lebesgue decomposition of normal positive functionals over $W^*(A)$ is unique. But Corollary \ref{fedoWun} asserts that this occurs exactly when the Lebesgue decomposition of positive functionals over $A$ is unique.
\end{proof}

In the terms of enveloping von Neumann algebras, the preceding theorem can be formulated into

\begin{corollary}\label{fedouniqe}
	Let $A$ be a $C^*$-algebra with enveloping von Neumann algebra $W^*(A)$.
	The following assertions are equivalent.
	\begin{itemize}
		\item[(i)]The Lebesgue decomposition of normal positive functionals over $W^*(A)$ is u\-ni\-qu\-e.
		\item[(ii)]$W^*(A)$ is a finite von Neumann algebra.
	\end{itemize}
\end{corollary}
\begin{proof}
	Corollary \ref{fedoWun}, Theorems \ref{mainC} and \ref{Hamana} show the equivalence.
\end{proof}

Now we examine the connections between the uniqueness over a 
$C^*$-algebra $A$ and the uniqueness over closed $^*$-subalgebras of $A$. 
As a special case, we consider closed ideals of $A$ (cf. 1.5.3 in \cite{Pedersen}). 

It is well known and not hard to obtain that if a $C^*$-algebra $A$ has only finite dimensional irreducible representations, then every closed $^*$-subalgebra $B$ of $A$ possesses this property. Indeed, Proposition 4.1.8 in \cite{Pedersen} states that if $(\rho, \mathcal{K})$ is an irreducible representation of a $C^*$-subalgebra $B$, then there is an irreducible representation $(\pi,\mathcal{H})$ of $A$ with a closed subspace $\mathcal{H}_1\subseteq\mathcal{H}$ such that $(\rho,\mathcal{K})$ is unitarily equivalent to $(\pi\!\mid\!_B,\mathcal{H}_1)$. Thus, the dimension of $\mathcal{K}$ is finite. 
Therefore, an instant consequence of Theorem \ref{mainC} is

\begin{proposition}\label{subalgend}
	If $A$ is a $C^*$-algebra, then the Lebesgue-decomposition of positive functionals over $A$ is unique if and only if the decomposition is unique over every closed $^*$-subalgebra $B$ of $A$.
\end{proposition}

Turning to closed ideals of $C^*$-algebras, we can make other interesting observations.
On the one hand, for a $C^*$-algebra $A$ with only finite dimensional irreducible representations it is true that every closed ideal $I$ of $A$ and the quotient $C^*$-algebra $A/I$ (cf. 1.5.5 in \cite{Pedersen}) have only finite dimensional irreducible representations (since such a representation on the latter algebra can be pulled back to an irreducible representation of $A$). On the other hand, if a $C^*$-algebra $A$ has a closed ideal $I$ such that $I$ and $A/I$ have only finite dimensional irreducible representations, then the dimensions of irreducible representations of $A$ are finite as well. Indeed, it is known that for an irreducible representation $\pi$ of $A$, the restriction $\pi\!\mid\!_I$ is zero or irreducible (\cite{Dixmier}, 2.11.3). Hence, if $\pi\!\mid\!_I\neq0$, then the finite dimensionality of this representation shows that the dimension of $\pi$ is finite too. If $\pi\!\mid\!_I=0$, then $\pi$ can be treated as a representation of $A/I$, thus, the finite dimensionality of $\pi$ also follows in this case. 
As a consequence, we get that every $C^*$-algebra contains a largest closed $^*$-ideal $I$ (possibly the zero ideal) such that all of the irreducible representations of $I$ are finite dimensional, and the quotient $C^*$-algebra $A/I$ has no closed ideals with only finite dimensional irreducible representations. (The ideal $I$ is just the sum of the closed $^*$-ideals with only finite dimensional irreducible representations.) 

According to our Theorem \ref{mainC}, these remarks can be presented in the context of the Lebesgue decomposition of positive linear functionals: for a $C^*$-algebra $A$, the decomposition over $A$ is unique if and only if there is a closed ideal $I$ of $A$ such that the decomposition is unique over $I$ and $A/I$. Furthermore, for any $C^*$-algebra $A$ there is a largest closed ideal $A_{LU}$ in $A$ such that the Lebesgue decomposition over $A_{LU}$ is unique, and in the $C^*$-algebra $A/A_{LU}$ the zero ideal is the only closed ideal with unique Lebesgue decomposition over it. 
In addition, the operation $A\mapsto A_{LU}$ has radical-like abilities similar to the cases of the so-called Leptin-radical (\cite{Palmer2}, 9.8.5 and 9.8.6) and Barnes-radical (\cite{Palmer2}, 10.5.21 and 10.5.23):
\begin{enumerate}
	\item[(1)] $(A/A_{LU})_{LU}=\{0\}$.
	
	\item[(2)] For any closed ideal $I\subseteq A$, $I_{LU}=I\cap A_{LU}$. 
In particular, $(A_{LU})_{LU}=A_{LU}$.
	
	\item[(3)] For every closed ideal $I\subseteq A$, if $(A/I)_{LU}=\{0\}$, then $I_{LU}=A_{LU}$.
	
	\item[(4)] $(A/I)_{LU}=A_{LU}/I$, if $I$ is a closed ideal in $A_{LU}$.
\end{enumerate} 
The proofs of these properties are fairly straightforward, so we omit the details.

Before proceeding to the general $^*$-algebra case, it is worth to ask some na\-tural questions.
\begin{question}\label{Q1}
Let $M$ be a von Neumann algebra.
\begin{itemize}
	\item[(1)] Is it true that if the Lebesgue decomposition of normal positive functionals over $M$ is unique, then the uniqueness holds in the case of all positive functionals over $M$?
	
	The answer is no in general. 
	If $M$ is a Type $II_1$ von Neumann algebra (e.g., \cite{Palmer2}, 9.3.27), then $M$ is finite, so the uniqueness over $M$ holds for normal positive functionals (Corollary \ref{normaluniq}). 
	However, regarding $M$ as a $C^*$-algebra, it has infinite dimensional irreducible representations, since every von Neumann algebra with only finite dimensional irreducible representations is of Type $I$. 
Hence, by Theorem \ref{mainC}, the uniqueness over $M$ is not true for all positive functionals.  

A more concrete finite Type $I$ example to this phenomena is the von Neumann algebra below, which is, in addition, an enveloping von Neumann algebra.
For a positive natural number $n$, let $(M_n(\mathbb{C}), \Vert\cdot\Vert_n)$ be the $C^*$-algebra of $n$ by $n$ matrices.
Let $A$ be the so-called $c_0$-direct sum of the $C^*$-algebra system $(M_n(\mathbb{C}))_{n\in\mathbb{N}^+}$, that is,
\[
A\cong\left\{(a_n)_{n\in\mathbb{N}^+}\left|a_n\in M_n(\mathbb{C}),\ \lim_{n\to+\infty}\Vert a_n\Vert_n=0 \right.\right\}.
\]
It is known that the enveloping von Neumann algebra of $A$ is $^*$-iso\-mor\-phic with the $\ell^{\infty}$-direct sum of the $C^*$-algebra system $(M_n(\mathbb{C}))_{n\in\mathbb{N}^+}$, that is,
\begin{equation*}\label{mproduct}
W^*(A)\cong\left\{(a_n)_{n\in\mathbb{N}^+}\left|a_n\in M_n(\mathbb{C}),\ \sup_{n\in\mathbb{N}}\Vert a_n\Vert_n<+\infty\right.\right\}.
\end{equation*}
On the one hand, $W^*(A)$ is a finite Type $I$ von Neumann algebra, so the uniqueness holds for normal positive functionals (Corollary \ref{normaluniq}).
On the other hand, as a $C^*$-algebra, $W^*(A)$ has infinite dimensional irreducible representations, since it is \emph{NGCR/antiliminal} by \cite{Kap}, Lemma 7.5.

	%\item[(2)] Assume in addition that $M=W^*(A)$ for some $C^*$-algebra $A$. Is it true now that if the Lebesgue decomposition of normal positive functionals over $M$ is unique, then the uniqueness holds in the case of all positive functionals over $M$?
	
	%\noindent The answer is still no. For every positive number $n$ let $M_n(\mathbb{C})$ be the $C^*$-algebra (von Neumann algebra) of $n$ by $n$ matrices. Let $A$ be the so-called $c_0$-direct sum of the $C^*$-algebra system $(M_n(\mathbb{C}))_{n\in\mathbb{N}^+}$, that is,
	%\[
	%A=\left\{(a_n)_{n\in\mathbb{N}^+}\left|a_n\in M_n(\mathbb{C}),\ \lim_{n\to+\infty}\Vert a_n\Vert=0 \right.\right\}.
	%\]
	%The enveloping von Neumann algebra of $A$ is the $\ell^{\infty}$-direct sum of the $C^*$-algebra system $(M_n(\mathbb{C}))_{n\in\mathbb{N}^+}$, that is,
	%\[
	%W^*(A)=\left\{(a_n)_{n\in\mathbb{N}^+}\left|a_n\in M_n(\mathbb{C}),\ \sup_{n\in\mathbb{N}}\Vert a_n\Vert<+\infty\right.\right\}.
	%\]
	%This is a finite (and Type $I$) von Neumann algebra (so the uniqueness holds for normal positive functionals), but has infinite dimensional irreducible representations. 
\end{itemize}

We do not know the answer for the two following questions in general.

\begin{itemize}
	\item[(2)] Is the converse of Corollary \ref{normaluniq} true? That is, if the uniqueness holds for normal positive functionals over $M$, is it true that $M$ is finite? And for $\sigma$-finite $M$?
	
	\item[(3)] From the proof of Corollary \ref{normaluniq}, it can be seen that if the assumption  
	\begin{equation}\label{WDtul2}
	L_{\mathfrak{g}}\subseteq L_{\mathfrak{f}}\ \Rightarrow\ \mathfrak{f}\ll\mathfrak{g}
	\end{equation}
	holds for every normal positive functionals $\mathfrak{f}$ and $\mathfrak{g}$ on a von Neumann algebra, then the uniqueness comes true for the Lebesgue decomposition in the normal case. Does the uniqueness imply this condition? Furthermore, can the finiteness of the algebra be derived from the property that \eqref{WDtul2} holds for every pair of normal positive functionals? 
	(For the latter question the answer is positive in the case of $\sigma$-finite von Neumann algebras, since Kosaki's Theorem \ref{Kosaki} (ii)$\Rightarrow$(i) shows this.)
\end{itemize}
\end{question}

\subsection{The case of general $^*$-algebras}\label{nemGNS}
To conclude the solution of the general uniqueness problem by the $C^*$-algebra case, we recall the concept of a \emph{$C^*$-seminorm} on a $^*$-algebra $A$ (see 9.5 in \cite{Palmer2}). 
It is a seminorm $\sigma:A\to\mathbb{R}_+$ which satisfies the $C^*$-property: $\sigma(a^*a)=\sigma(a)^2$ for every $a\in A$. 
A remarkable theorem of Z. Sebestyén (\cite{Se}, or Theorem 9.5.14 in \cite{Palmer2}) shows that a $C^*$-seminorm $\sigma$ on $A$ is actually submultiplicative, i. e., $\sigma(ab)\leq\sigma(a)\sigma(b)$ for any $a,b\in A$. 
Moreover, the equation $\sigma(a^*)=\sigma(a)$ holds for every $a\in A$.
Note that every $^*$-algebra admits a $C^*$-seminorm (namely, the identically zero seminorm).

For a $C^*$-seminorm $\sigma$ on $A$, the \emph{kernel} of $\sigma$, that is, the $^*$-ideal
\[
\{a\in A|\sigma(a)=0\}
\]
is denoted by $\ker(\sigma)$. 
If $\ker(\sigma)=\{0\}$, i.e., $\sigma$ is a (possibly incomplete) norm, then we say
that $\sigma$ is a \emph{pre-$C^*$-norm} and $(A, \sigma)$ is a pre-$C^*$-algebra.

By Theorem 9.5.17 in \cite{Palmer2}, a positive linear functional $f:A\to\mathbb{C}$  is representable if and only if there is a $C^*$-seminorm $\sigma$ on $A$ such that $f$ is continuous with respect to $\sigma$. That is, there exists a $K\in\mathbb{R}_+$ such that
\[
\lvert f(a)\rvert\leq K\sigma(a)\ \ \ \ (\forall a\in A).
\]

The following observation is useful.

\begin{remark}\label{represent}
	Let $A$ be a $^*$-algebra and assume that $t,f,g:A\to\mathbb{C}$ are representable positive functionals on $A$.
	By the theorem mentioned above, there are $C^*$-seminorms $\sigma_t, \sigma_f$ and $\sigma_g$ on $A$ and there exist non-negative numbers $K_t, K_f, K_g$ such that the inequalities
	\[
	\lvert t(a)\lvert\leq K_t\sigma_t(a);\ \lvert f(a)\rvert\leq K_f\sigma_f(a);\ \lvert g(a)\rvert\leq K_g\sigma_g(a) 
	\]
	hold for any $a\in A$. With the number $K:=\max\{M_t,M_f,M_g\}$ and the $C^*$-seminorm $\sigma:=\max\{\sigma_t, \sigma_f, \sigma_g\}$, we have for every $a\in A$
	\[
	\lvert t(a)\rvert\leq K\sigma(a);\ \lvert f(a)\rvert\leq K\sigma(a);\ \lvert g(a)\rvert\leq K\sigma(a), 
	\]
	that is, the functionals are continuous with respect to $\sigma$.
\end{remark}

\begin{lemma}\label{faktoruniq}
	Let $A$ be a $^*$-algebra and let $I\subseteq A$ be a $^*$-ideal. 
	If the Lebesgue decomposition of representable positive functionals is unique over $A$, then the same is true over the quotient $^*$-algebra $A/I$.
\end{lemma}
\begin{proof}
	We verify the condition in Lemma \ref{uniqlem1} (ii), so consider representable positive functionals $t',f'$ and $g'$ on $A/I$ such that
	\[
	t'\leq f'\ll g'.
	\]
	To see that $t'\ll g'$, denote by $j$ the canonical $A\to A/I$ surjective $^*$-homorphism (i.e., $j(a):=a+I$), and let
	\[
	t=t'\circ j;\ f=f'\circ j;\ g=g'\circ j. 
	\]
	It is obvious that these mappings are positive functionals on $A$, moreover they are representable.
	Indeed, by the preceding Remark \ref{represent}, there is a $C^*$-seminorm $\sigma'$ on $A/I$ with $K\in\mathbb{R}_+$ such that
	\[
	\lvert t'(a+I)\rvert\leq K\sigma'(a+I);\ \lvert f'(a+I)\rvert\leq K\sigma'(a+I);\ \lvert g'(a+I)\rvert\leq K\sigma'(a+I) 
	\]
	is true for any $a\in A$. Using the mapping $j$ in the formulation, we infer that
	\[
	\lvert t'(j(a))\rvert\leq K\sigma'(j(a));\ \lvert f'(j(a))\rvert\leq K\sigma'(j(a));\ \lvert g'(j(a))\rvert\leq K\sigma'(j(a)), 
	\]
	which exactly means that the functionals $t,f,g$ are continuous with respect to the $C^*$-seminorm $\sigma:=\sigma'\circ j$.
	
	From Definition \ref{absdef} of absolute continuity, it is easy to see that
	\[
	t\leq f\ll g
	\]
	holds (the inequality $t\leq f$ is trivial).
	Hence, the uniqueness over $A$ provides that $t\ll g$, which immediately shows that $t'\ll g'$. 
\end{proof}

\begin{remark}\label{easy}
	An observation similar to the preceding proof is the following. 
	Let $A$ be a $^*$-algebra and let $t,f,g$ be representable positive functionals on $A$.
	Assume that the functionals are continuous with respect to the $C^*$-seminorm $\sigma:A\to\mathbb{R}_+$. 
Denote  by $I$ the kernel of the seminorm, that is, $I:=\ker(\sigma)$.
	Then the functionals 
	\[
	t'(a+I):=t(a);\ f'(a+I):=f(a);\ g'(a+I):=g(a)\ \ \ (a\in A)
	\] 
	are well-defined, positive linear and representable on the quotient $^*$-algebra $A/I$. In fact, this $^*$-algebra possesses the quotient pre-$C^*$-norm $\sigma'$ induced by $\sigma$, and the functionals are continuous with respect to this norm.
	
	The following equivalence, which we use later, can be easily obtained from the definitions:
	\begin{equation}\label{kamu}
	(t\leq f\ll g)\ \Leftrightarrow\ (t'\leq f'\ll g').
	\end{equation}
\end{remark}

The next statement is a special case of  Lemma 2.10 in \cite{SzAbs}.
We note that if $f$ is a positive functional on a $C^*$-algebra $\mathscr{A}$, then every restriction of $f$ to a $^*$-subalgebra $\mathscr{D}$ is obviously representable, since $f$ is continuous with respect to the $C^*$-norm.  

\begin{lemma}\label{surus}
	Let $\mathscr{A}$ be a $C^*$-algebra with a dense $^*$-subalgebra $\mathscr{D}$.
	If $f$ and $g$ are positive functionals on $\mathscr{A}$, then
	\begin{enumerate}
		\item[(1)] $f\leq g$ if and only if $f\!\mid\!_{\mathscr{D}}\leq g\!\mid\!_{\mathscr{D}}$.
		
		\item[(2)] $f$ is absolutely continuous with respect to $g$ if and only if $f\!\mid\!_{\mathscr{D}}$ is absolutely continuous with respect to $g\!\mid\!_{\mathscr{D}}$.
	\end{enumerate}
\end{lemma}

The last essential observation is the following.

\begin{corollary}\label{vegsocor}
	If $\mathscr{A}$ is a $C^*$-algebra with a dense $^*$-subalgebra $\mathscr{D}$, then the uniqueness of the Lebesgue decomposition of representable positive functionals over $\mathscr{D}$ implies the uniqueness of the decomposition of positive functionals over $\mathscr{A}$.
\end{corollary}
\begin{proof}
	We use Lemma \ref{uniqlem1} (ii) to prove the uniqueness over $\mathscr{A}$, so let $t,f,g$ be positive functionals on $\mathscr{A}$ with the condition
	\[
	t\leq f\ll g.
	\] 
	From the previous Lemma \ref{surus}, we conclude that
	\[
	t\!\mid\!_{\mathscr{D}}\leq f\!\mid\!_{\mathscr{D}}\ll g\!\mid\!_{\mathscr{D}}.
	\]
	Thus, the uniqueness over $\mathscr{D}$ implies that $t\!\mid\!_{\mathscr{D}}\ll g\!\mid\!_{\mathscr{D}}$.
	Using Lemma \ref{surus} again, we obtain that $t\ll g$, that is, the uniqueness over $\mathscr{A}$ follows.
\end{proof}

All of the necessary results and tools are on hands now, so we prove our main Theorem \ref{main} on the connection of the uniqueness of the Lebesgue decomposition and finite dimensional irreducible representations. 
We restate it here for the sake of clarity.

\begin{theorem}\label{mainnew}
	For a $^*$-algebra $A$, the following statements are equivalent.
	\begin{itemize}
		\item[(i)]The Lebesgue decomposition of representable positive functionals over $A$ is u\-ni\-qu\-e.
		\item[(ii)]Every topologically irreducible representation of $A$ is finite dimensional.
	\end{itemize}
\end{theorem}
\begin{proof}
According to Remarks \ref{trivial} and \ref{trivial2}, the set of topologically irreducible representations of $A$ is void iff the zero functional is the only representable positive functional on $A$. 
Hence, in this case, the equivalence of (i) and (ii) is obvious. 

For the rest of the proof, we assume that there is a topologically irreducible representation of $A$.
	
(i)$\Rightarrow$(ii): Consider a topologically irreducible representation $\pi:A\to\mathscr{B}(\mathcal{H})$ of $A$ on the Hilbert space $\mathcal{H}$.
	Let $\mathscr{D}$ be the range $^*$-algebra $\pi\left<A\right>$, and denote by $\mathscr{A}$ the norm closure   
	of $\mathscr{D}$ in $\mathscr{B}(\mathcal{H})$. 
	
	The identity representation of the $C^*$-algebra $\mathscr{A}$ on $\mathcal{H}$ is irreducible, thanks to the topological irreducibility of $\pi$. 
	Hence, if we show that the Lebesgue decomposition of positive functionals over $\mathscr{A}$ is unique, then Theorem \ref{mainC} forces the dimension of $\mathcal{H}$ to be finite, providing (ii). But the uniqueness comes easily from the lemmas: in virtue of (i) and Lemma \ref{faktoruniq}, the uniqueness holds for representable positive functionals over $\mathscr{D}$, so the preceding Corollary \ref{vegsocor} proves the uniqueness over $\mathscr{A}$.

	(ii)$\Rightarrow$(i):
	Suppose that every topologically irreducible representation of $A$ is finite dimensional. To check the uniqueness, let $t,f,g$ be representable positive functionals on $A$ such that 
	\[
	t\leq f\ll g.
	\] 
	In accordance to Lemma \ref{uniqlem1} (ii), we need to show that $t\ll g$. 
	The representability implies the existence of a $C^*$-seminorm $\sigma$ on $A$ such that the three functionals are continuous with respect to $\sigma$ (Remark \ref{represent}). Let $I$ be the kernel of $\sigma$. Denote by $\mathscr{D}$ the quotient $^*$-algebra $A/I$ and let $j$ be the canonical $A\to A/I$ surjection.
	Consider the positive functionals $t',f',g'$ and the pre-$C^*$-norm $\sigma'$ on $A/I=\mathscr{D}$ according to Remark \ref{easy}.  
	By \eqref{kamu}, we have
	\[
	(t\leq f\ll g)\ \Leftrightarrow\ (t'\leq f'\ll g').
	\]
	Let $\mathscr{A}$ stand for the completion of the pre-$C^*$-algebra $(\mathscr{D},\sigma')$. Since the positive functionals are continuous (with respect to the $C^*$-norm) on the dense $^*$-algebra $\mathscr{D}$, they admit unique positive linear extensions $T',F',G'$ to the whole $C^*$-algebra $\mathscr{A}$.
	Now Lemma \ref{surus} infers that
	\[
	(t'\leq f'\ll g')\ \Leftrightarrow\ (T'\leq F'\ll G').
	\] 
	Thus, if we show that $T'\ll G'$, then using Lemma \ref{surus} again, $t'\ll g'$ follows. Consequently, $t\ll g$ by means of Remark \ref{easy}.
	
	To claim that $T'$ is absolutely continuous with respect to $G'$, it is enough to prove that every irreducible representation of the $C^*$-algebra $\mathscr{A}$ is finite dimensional. 
	Indeed, in this case Theorem \ref{mainC} provides that the Lebesgue decomposition of positive functionals over $\mathscr{A}$ is unique, thus, from the property $T'\leq F'\ll G'$,  Lemma \ref{uniqlem1} forces that $T'\ll G'$.
	But, if $\pi:\mathscr{A}\to\mathscr{B}(\mathcal{H})$ is an irreducible representation, then density implies that $\pi\!\mid\!_{\mathscr{D}}$ is a topologically irreducible representation of $\mathscr{D}$. 
Hence, the surjectivity of $j$ implies that $\left(\pi\!\mid\!_{\mathscr{D}}\right)\circ j:A\to\mathscr{B}(\mathcal{H})$ is a topologically irreducible representation of $A$. From the assumption (ii), it is finite dimensional, consequently $\pi$ is finite dimensional. 
	The proof is complete.  
\end{proof}

\section{REMARKS AND EXAMPLES}\label{utols}

The last section of the paper contains interesting examples and remarks re\-le\-vant to the Lebesgue decomposition theory of positive functionals and our characterization Theorem \ref{mainnew}. 

To throw more light on the non-$C^*$-algebra version of the Lebesgue decomposition and the uniqueness, we collect several facts and examples from the theories of general $^*$-algebras (\cite{Palmer2}) and topological $^*$-algebras (\cite{F}).
We discuss the question here that which kind of $^*$-al\-geb\-ras have non-trivial Lebesgue decomposition theory for representable positive functionals. 
Moreover, if possible, one might require that every element of the $^*$-algebra would be involved in the investigations.

Let $A$ be a $^*$-algebra. 
Following T.W. Palmer (\cite{Palmer2}, section 9.7), we recall the next

\begin{definition}
Denote by $A_R$ the intersection of the kernels of all of the ($^*$-)representations of $A$.
The $^*$-ideal $A_R$ is the \emph{reducing ideal} of $A$. 
If $A_R=\{0\}$ (resp. $A_R=A$), then $A$ is said to be \emph{reduced} (resp. \emph{$^*$-radical}).
\end{definition} 

Theorem 9.7.2 in \cite{Palmer2} shows that 
\begin{equation}\label{redid}
A_R=\bigcap_{\textrm{$\sigma$ is a $C^*$-seminorm on $A$}}\ker(\sigma),
\end{equation}
and, in fact, $A_R$ is the intersection of the kernels of the representable positive functionals on $A$. 
Furthermore, the reducing ideal equals to the intersection of the kernels of all of the topologically irreducible representations of $A$.
In particular, $A=A_R$ iff $A$ does not admit such a representation, and this occurs exactly when the only representable positive functional on $A$ is the zero functional (Remark \ref{trivial2}).
Hence, to get non-trivial Lebesgue decomposition theory, we seek $^*$-algebras which are reduced, i.e., there are enough representable positive functionals/topologically irreducible representations to separate the points of the algebra in question.
Note that the Jacobson-radical $A_J$ of $A$ (\cite{Palmer1}, Definition 4.3.1 and Theorem 4.3.6) is always contained in the reducing ideal $A_R$ (\cite{Palmer2}, Theorem 9.7.11), so a reduced $^*$-algebra is (Jacobson-)semisimple (i.e., $A_J=\{0\}$), while a Jacobson-radical $^*$-algebra (i.e., $A_J=A$) is $^*$-radical.

By \eqref{redid}, it is obvious that pre-$C^*$-algebras (in particular, $C^*$-algebras) are reduced. 
In Subsections \ref{redGcsillag} and \ref{rednonGcsillag}, we list many different types of reduced $^*$-algebras which are not $C^*$-algebras. 
On the other extreme, we must record here that there are many interesting $^*$-algebras which are $^*$-radicals.
Some of them are included in the next

\begin{example}\label{trivipl}
The trivial example for a $^*$-radical $^*$-algebra is an arbitrary complex vector space $A$ with an involution and zero multiplication.
By definition, every linear functional on $A$ is positive. 
On the other hand, the zero functional is the only representable positive functional on $A$.
Indeed, if $f$ is a representable positive functional on $A$, then by Theorem 9.4.15 in \cite{Palmer2}, there exists a $K\in\mathbb{R}_+$ such that
\[
|f(a)|^2\leq K f(a^*a)
\]
holds for any $a\in A$. 
Hence $f$ must be zero.

The following non-trivial examples are from Palmer's book \cite{Palmer2} (the numbers refer to this work).  
\begin{itemize}
\item[(9.7.27)] The $^*$-algebra of all complex $\mathbb{N}\times\mathbb{N}$ matrices that have only finitely many non-zero entries in each row and column is a unital, semisimple, and $^*$-radical $^*$-algebra.   

\item[(9.7.29)] The $^*$-algebra of all complex rational functions is a field (hence simple), which is a $^*$-radical.
In fact, any positive linear functional on this $^*$-algebra is zero.

\item[(9.7.31)] The convolution disc $^*$-algebra is a Jacobson-radical (hence $^*$-radical) commutative Banach $^*$-algebra.
Note that this algebra has no non-zero divisors of zero (\cite{Palmer1}, 4.8.3).

\item[(9.7.33)] The complex matrix algebras $M_{n,k}$, endowed with the exotic involutions described preceding to Theorem 9.1.46 in \cite{Palmer2}, are also $^*$-radicals. (As the referee pointed out, similar involutions can be defined on the Banach algebra $\mathscr{B}(\mathcal{H})$ for an infinite dimensional Hilbert space $\mathcal{H}$.)
\end{itemize}
\end{example}

Now let $A$ be a general $^*$-algebra.
We have seen that the proof of Theorem \ref{mainnew} on the uniqueness of the Lebesgue decomposition over $A$ lies upon the $C^*$-algebra version of the uniqueness (Theorem \ref{mainC}).
We used factorization arguments (see the lemmas and remarks in Subsection \ref{nemGNS}) to transfer positive functionals and representations to certain $C^*$-algebras.
Similarly to the latter arguments, the existence theorem on the Lebesgue decomposition of representable positive functionals over $A$ (Theorem \ref{Lebesgue}) can be obtained through the $C^*$-algebra version of the existence (in fact, by Lemma \ref{abszlem} and Remark \ref{messe}, the von Neumann algebra version is sufficient).
However, if we look at the proof of Theorem \ref{mainnew}, then we see that the transferring methods issue different (probably infinitely many) $C^*$-algebras for different functionals (and representations). 
So, one might ask the following. 

\begin{question}\label{quest}
Is there a distinguished $C^*$-algebra $C^*(A)$ such that the Le\-bes\-gue decomposition theory of representable positive functionals over a reduced $^*$-algebra $A$, including the uniqueness property, is essentially the same as that of the Le\-bes\-gue decomposition theory of positive functionals over $C^*(A)$? 
\end{question}

After Proposition \ref{bhatt}, we point out that the answer is positive only for a special class of $^*$-algebras, the so-called $G^*$-algebras (\cite{Palmer2}, section 10.1). 
Thus, according to the discussion above, we handle two cases  in the context of reduced $^*$-algebras: $G^*$-algebras and non-$G^*$-algebras. 
The latter case involves a class of topological $^*$-algebras, the so-called locally $C^*$-algebras (\cite{F}, chapter II).
In the end of Subsection \ref{rednonGcsillag}, we show for a reduced $^*$-algebra $A$ that the Lebesgue decomposition theory of representable positive functionals over $A$ is essentially the same as that over a distinguished locally $C^*$-algebra $E(A)$.

\subsection{Reduced $G^*$-algebras}\label{redGcsillag}

First we recall the concepts of $G^*$-al\-geb\-ras and the Gelfand-Naimark seminorm (\cite{Palmer2}, Definition 10.1.1).
\begin{definition}
Let $A$ be a $^*$-algebra. 
For every $a\in A$, let
\begin{equation*}
\gamma_A(a):=\sup\{\Vert\pi(a)\Vert|\pi\ \textrm{is a $^*$-representation of $A$ on a Hilbert space}\}.
\end{equation*}
If $\gamma_A$ is finite valued, then it is called the \emph{Gelfand-Naimark seminorm of $A$}, and we say that $A$ is a \emph{$G^*$-algebra}.
\end{definition}

By Proposition 10.1.2 in \cite{Palmer2}, it follows for any $^*$-algebra $A$ and $a\in A$ that
\begin{equation}\label{greatest}
\gamma_A(a)=\sup\{\sigma(a)|\textrm{$\sigma$ is a $C^*$-seminorm on $A$}\}.
\end{equation}
If $A$ is a $G^*$-algebra, then $\gamma_A$ is the greatest $C^*$-seminorm on $A$ and $A_R=\ker\gamma_A$. 
Thus, a $G^*$-algebra $A$ is reduced if and only if $\gamma_A$ is a norm. 
Moreover, a positive functional on $A$ is representable exactly when it is continuous with respect to $\gamma_A$ (\cite{Palmer2}, Theorem 10.1.3 and Corollary 10.1.8).

Every Banach $^*$-algebra is a $G^*$-algebra (e.g., by Theorem 11.1.5 in \cite{Palmer2}), as well as the $^*$-algebras in Example \ref{trivipl}. 
However, we are interested in reduced examples. 
For such a $^*$-algebra, we recall the notion and the properties of the enveloping $C^*$-algebra (\cite{Palmer2}, 10.1.10-12).

\begin{definition}\label{fedoCalg}
Let $A$ be a reduced $G^*$-algebra. 
Denote by $C^*(A)$ the completion of $A$ with respect to the pre-$C^*$-norm $\gamma_A$.
Then the $C^*$-algebra $C^*(A)$  is called the \emph{enveloping $C^*$-algebra of $A$}.   
\end{definition}

Theorem 10.1.12 in \cite{Palmer2} shows for a $^*$-algebra $A$ such that representable positive functionals and the $^*$-representation theory of $A$ are essentially the same as that of its enveloping $C^*$-algebra $C^*(A)$. 
From this, it was shown in Subsection 5.1 of \cite{SzAbs} that the Lebesgue decomposition theory of representable positive functionals over $A$ coincides with the Lebesgue decomposition theory of positive func\-ti\-o\-nals o\-ver $C^*(A)$ (this also follows from the discussions in Subsection \ref{rednonGcsillag} below).
Hence, in the case of reduced $G^*$-algebras, Question \ref{quest} has a positive answer.

Now we present some examples of reduced $G^*$-algebras.

\begin{example}
Let $(\mathscr{A},\Vert\cdot\Vert)$ be an (infinite dimensional) \emph{approximately finite dimensional $C^*$-algebra} in the following sense (\cite{Ta3}, chapter XIX, § 1):
there exists an increasing sequence $(A_n)_{n\in\mathbb{N}}$ consisting of finite dimensional $^*$-sub\-al\-geb\-ras of $\mathscr{A}$ such that
\[
\mathscr{A}=\overline{\bigcup_{n\in\mathbb{N}}A_n}.
\]
By \eqref{redid}, the $^*$-algebra $A:=\cup_{n\in\mathbb{N}}A_n$ is reduced, since the restriction $\Vert\cdot\Vert|_A$ is a pre-$C^*$-norm on $A$. In fact, $\Vert\cdot\Vert|_A$ is the only pre-$C^*$-norm on $A$.
Indeed, let $\sigma:A\to\mathbb{R}_+$ be a pre-$C^*$-norm and let $a\in A$ be an arbitrary element.
There exists an $n\in\mathbb{N}$ such that $a\in A_n$.
The finite dimensionality of $A_n$ implies that the pre-$C^*$-norms $\Vert\cdot\Vert|_{A_n}$ and $\sigma|_{A_n}$ are actually complete. 
It is well known that there is at most one $C^*$-norm on a $^*$-algebra (\cite{Palmer2}, Proposition 10.1.9), so $\sigma(a)=\Vert a\Vert$ follows.
Thus, $\gamma_A(a)=\Vert a\Vert$, and by \eqref{greatest}, $A$ is $G^*$-algebra. 
A concrete example for such an $A$ is the $^*$-algebra of finitely supported complex sequences with the pointwise operations and the norm $\Vert(x_n)_{n\in\mathbb{N}}\Vert:=\sup_{n\in\mathbb{N}}|x_n|$.

Other examples for reduced $G^*$-algebras can be found in \cite{Palmer2}: 10.3.17, 10.3.18 and the $^*$-algebras that appear in Theorem 10.6.11.
Excluding the cases of the latter, there are no complete submultiplicative norms on these $^*$-algebras. 
\end{example}

As we mentioned above, every Banach $^*$-algebra $A$ is a $G^*$-algebra. 
Hence, by \eqref{redid}, $A$ is reduced if and only if there is a pre-$C^*$-norm on $A$.
In the earlier literature, reduced Banach $^*$-algebras were called \emph{$A^*$-algebras}.

\begin{example}\label{Acsillag}
Below we list some examples of reduced Banach $^*$-algebras which are not $C^*$-algebras.
\begin{itemize}
\item[(1)] The disc $^*$-algebra $A(\mathbb{D})$ on the closed unit disc $\mathbb{D}$  (\cite{Palmer2}, 9.7.25).
That is, endowed with the pointwise operations and the supremum norm, $A(\mathbb{D})$ is the Banach algebra of continuous complex valued functions on $\mathbb{D}$ that are analytic on the interior of $\mathbb{D}$, equipped with the involution
\[
^*:A(\mathbb{D})\to A(\mathbb{D});\ a^*(z)=\overline{a(\overline{z})}\ (z\in\mathbb{D}).
\]  
The enveloping $C^*$-algebra of $A(\mathbb{D})$ is the $C^*$-algebra of continuous functions on the interval $[-1,1]$.

\item[(2)] The Banach $^*$-algebras of the Hilbert-Schmidt and the trace-class operators on an infinite dimensional Hilbert space $\mathcal{H}$ (\cite{Palmer2}, Theorems 9.1.32. and 9.1.35) are not $C^*$-algebras. 
The enveloping $C^*$-algebra of both $^*$-algebras is the $C^*$-al\-geb\-ra of the compact operators on $\mathcal{H}$. 
\end{itemize}
The Lebesgue decompositions over the $^*$-algebras in (1) is unique, while non-uniqueness occurs for the $^*$-algebras in (2).
\end{example}

Because of its importance, we treat the case of the $L^1$-algebras of locally compact groups separately. 
As we mentioned at the end of the Preliminaries, the uniqueness of the Lebesgue decomposition provides a new characterization of Moore groups. These groups were named after C. C. Moore, due to his famous article \cite{Mo}, in which many beautiful properties and characterizing criterions were proved. 
For a comprehensive description, see sections 12.4 and 12.5 in Palmer's book \cite{Palmer2}. 
The basic definition of such a locally compact group $G$ is the following. If $V:G\to\mathscr{B}(\mathcal{H})$ is a continuous, unitary representation on the Hilbert space $\mathcal{H}$ which is topologically irreducible (\cite{Palmer2}, Definition 9.2.19), then $\mathcal{H}$ is finite dimensional. Hence finite groups, compact groups and locally compact abelian groups belong to Moore groups.

Let $G$ be a locally compact group.
For a fixed left Haar measure $\beta$ on $G$ the space $L^1(G)$ is a Banach $^*$-algebra with the convolution product related to $\beta$ and a natural isometric involution (\cite{Dixmier}, 13.2; \cite{Palmer2}, 9.1.8).
Moreover, well-known results (\cite{Palmer2}, Theorem 12.4.4 and Corollary 12.4.5) show that there exists a pre-$C^*$-norm on $L^1(G)$, hence it is a reduced Banach $^*$-algebra.
The enveloping $C^*$-algebra of $L^1(G)$ is the so-called \emph{full group-$C^*$-algebra} (with common notation: $C^*(G)$).
It is a remarkable fact that  every positive functional defined on $L^1(G)$ is continuous and representable (\cite{Dixmier}, B29 and 13.2.5; \cite{Palmer2}, Corollary 11.3.8). 

Theorem 12.4.1 in \cite{Palmer2} describes a natural bijective correspondence between the continuous unitary representations of $G$   and the non-degenerate representations of $L^1(G)$ (as well as the non-degenerate representations of  $C^*(G)$). By this bijection, one can point out that $G$ is a Moore group if and only if every topologically irreducible representation of $L^1(G)$ is finite dimensional. Hence, our main Theorem \ref{mainnew} and the previous comments immediately imply the following.

\begin{corollary}\label{Moore}
Let $G$ be a locally compact group. Then the following statements are equivalent.
\begin{itemize}
\item[(i)] $G$ is a Moore group.
\item[(ii)] The Lebesgue decomposition of positive functionals over $L^1(G)$ is unique.
\item[(iii)] The Lebesgue decomposition of positive functionals over $C^*(G)$ is unique.
\end{itemize}
\end{corollary}

\subsection{Reduced $^*$-algebras which are not $G^*$-algebras}\label{rednonGcsillag}
In this subsection, we refer the reader mainly to \cite{F} and \cite{Palmer2}. 
Let us begin with examples.

\begin{example}\label{produktum}
For every positive integer $n$, denote by $(M_n(\mathbb{C}), \Vert\cdot\Vert_n)$ the $C^*$-algebra of $n$ by $n$ matrices.
Let $A_M$ be the complete direct product $^*$-algebra of the $M_n(\mathbb{C})$'s, endowed with the coordinatewise operations, i.e.,
\[
A_M:=\prod_{n\in\mathbb{N}^+}M_n(\mathbb{C})=\left\{(a_n)_{n\in\mathbb{N}^+}\left|a_n\in M_n(\mathbb{C})\right.\right\}.
\]
It is not hard to obtain that a non-zero $C^*$-seminorm on $A_M$ has the form 
\begin{equation}\label{maxnorm}
\sigma_F:A_M\to\mathbb{R}_+;\ \sigma_F((a_n)_{n\in\mathbb{N}^+})=\sup_{n\in F}\Vert a_n\Vert_n,
\end{equation}
where $F$ is a non-void finite subset of $\mathbb{N}^+$.
By Theorem 9.5.4 in \cite{Palmer2}, for every $C^*$-seminorm $\sigma$ on $A_M$ there is a $^*$-representation $\pi$ of $A_M$ such that $\sigma(a)=\Vert\pi(a)\Vert$ for all $a\in A_M$. 
Now, by the very similar calculations as in 9.7.26 and 9.7.27 of \cite{Palmer2}, we conclude that there is a finite subset $F$ of $\mathbb{N}^+$ such that
\[
\ker(\pi)\supseteq\left\{(a_n)_{n\in\mathbb{N}^+}\in A\left|\forall\ n\in F:\ a_n=0 \right.\right\}.
\]
Indeed, if $\mathbf{1}_n$ denotes the identity matrix in $M_n(\mathbb{C})$, then for every $m\in\mathbb{N}^+$ consider the element $e_m=(e_{mn})_{n\in\mathbb{N}^+}\in A$ such that
\[
 e_{mn}= \left\{\begin{array}{ll}
        0, & \text{if } 1\leq n\leq m;\\
        \mathbf{1}_n, & \text{if } m<n. 
        \end{array}\right.
\] 
Let $a=(n\mathbf{1}_n)_{n\in\mathbb{N}^+}$.
In the natural ordering of self-adjoint operators, the elements above satisfy $\pi((m+1)e_m)\leq\pi(a)$ $(m\in\mathbb{N}^+)$, thus $(m+1)\Vert\pi(e_m)\Vert\leq\Vert\pi(a)\Vert$. 
But $\pi(e_m)$ is a projection, so its norm must be $0$ or $1$.
From this we get that $\pi(e_m)=0$ for all but finitely many $m\in\mathbb{N}^+$.
Now \eqref{maxnorm} follows.
In fact, we obtained that the range of every $^*$-representation $\pi$ is finite dimensional, and, in particular, every topologically irreducible representation of $A_M$ is finite dimensional. (Such a representation can be obtained by the "projection" onto the $n$th coordinate.) It follows that the Lebesgue decomposition over $A_M$ is unique (Theorem \ref{mainnew}).

Hence, there is no pre-$C^*$-norm on $A_M$. 
By \eqref{redid}, $A_M$ is a reduced (the intersection of the kernels of all $C^*$-seminorms is $\{0\}$).
Furthermore, $A_M$ is not a $G^*$-algebra (the element $a$ above has $\gamma_{A_M}(a)=+\infty$).
\end{example}

\begin{example}\label{polialg}
Let $A$ be the $^*$-algebra of all functions $a:\mathbb{R}\to\mathbb{C}$ that are polynomials, endowed with pointwise operations and complex conjugation as the involution.
Then every non-zero $C^*$-seminorm on $A$ is of the form
\[
\sigma_K:A\to\mathbb{R}_+;\ \sigma_K(a)=\sup_{t\in K}|a(t)|,
\]
where $K$ is an arbitrary non-void compact subset of $\mathbb{R}$.
For an infinite $K$, $\sigma_K$ is a pre-$C^*$-norm, hence $A$ is reduced.
Moreover, for a non-constant $a$ we have that
\[
\gamma_A(a)=\sup\{\sigma_K(a)|\textrm{$K\subseteq\mathbb{R}$, $K$ is non-void and compact}\}=+\infty. 
\]
Thus, $A$ is not a $G^*$-algebra (but admits many pre-$C^*$-norms).
Note here that the Lebesgue decomposition over $A$ is unique, since every topologically irreducible representation of $A$ is one dimensional (\cite{Palmer2}, Theorem 9.6.10). 
\end{example}

Now let $A$ be a reduced $^*$-algebra.
As we mentioned after \eqref{greatest}, if $A$ is a $G^*$-algebra, then a positive functional $f$ on $A$ is representable iff $f$ is continuous with respect to the Gelfand-Naimark norm $\gamma_A$.  
However, if $A$ is not a $G^*$-algebra, a concrete locally convex Hausdorff topology $\tau_R$ can be defined on $A$ such that the representable positive functionals on $A$ are exactly the $\tau_R$-continuous positive functionals (Theorem \ref{csillagR}).
We exhibit this below.
Moreover, it turns out that the Lebesgue decomposition theory of representable positive functionals over $A$ is actually the same as the theory over a distinguished topological $^*$-algebra $E(A)$ (Definition \ref{fedolok}, discussion after Proposition \ref{bhatt}).

For the general theory of topological algebras and $^*$-algebras we refer the reader to A. Mallios' work \cite{Ma} and  M. Fragoulopoulou's book \cite{F}.
Let $A$ be an algebra. 
A \emph{locally $m$-convex topology} $\tau=\tau_{\Gamma}$ on $A$ is a topology which is defined by a directed/saturated (\cite{F}, page 7) family 
$\Gamma=\{\sigma_i|i\in \mathfrak{I}\}$
of submultiplicative seminorms on $A$. 
The algebra $A$, equipped with $\tau$, is a \emph{locally $m$-convex algebra}, which is denoted by $A[\tau]$.
It  is Hausdorff iff $\Gamma$ separates the points of $A$.
Moreover, $A[\tau]$ has a (jointly) continuous multiplication.
We call metrizable complete locally $m$-convex algebras \emph{Fréchet locally $m$-convex algebras}. 
If $A[\tau_{\Gamma}]$ is a locally $m$-convex algebra which is also a $^*$-algebra, and in addition every $\sigma\in\Gamma$ satisfies $\sigma(a^*)=\sigma(a)$ $(a\in A)$, then we use the term \emph{locally $m$-convex $^*$-algebra}. 

The following is Definition 7.5 on page 102 in \cite{F} (see also: \cite{Apostol}, \cite{Inoue}, \cite{Phil}). 

\begin{definition}
A Hausdorff locally $m$-convex $^*$-algebra $A[\tau]$, whose to\-po\-logy $\tau=\tau_{\Gamma}$ is determined by a family $\Gamma=\{\sigma_i|i\in \mathfrak{I}\}$  of $C^*$-seminorms, is called \emph{$C^*$-convex algebra}. 
If $A[\tau]$ is complete, then it is called \emph{locally $C^*$-algebra}.
\end{definition}

\begin{remark}
By definition, a $C^*$-convex algebra $A[\tau_{\Gamma}]$ is reduced,
and it has continuous multiplication and involution.
Hence, its (Hausdorff) completion 
$
\widetilde{A}[\tau_{\widetilde{\Gamma}}]
$ (\cite{F}, page 8) 
is a locally $C^*$-algebra.
The topology of $\widetilde{A}$ is determined by the family $\widetilde{\Gamma}$ consisting the unique continuous extensions of the elements of $\Gamma$ to $\widetilde{A}$. 
\end{remark}

According to \cite{F} (page 99, Introduction to Chapter II), the term "locally $C^*$-algebra" is due to A. Inoue (\cite{Inoue}).
Other names, such as "$b^*$-algebras" (e.g., \cite{Apostol}) or "pro-$C^*$-algebras" (e.g., \cite{Phil}) have been appeared for this kind of (generally non-normed) $^*$-algebras. 
There is an extensive literature dealing with their theory, in particular, the Chapters II-IV in \cite{F} (see also the above mentioned works and the Introduction to Chapter II in \cite{F}).
These references also contain many applications of the class of locally $C^*$-algebras.

Before providing some examples, our purposes need a concrete locally $m$-convex topology, which can be naturally defined on an arbitrary *-algebra $A$, and which makes each representable positive functional on $A$ to be continuous. 
So, we can involve this topology in the investigations of the Lebesgue decomposition theory. 
We treat only the reduced case.

Following Palmer \cite{Palmer2}, we introduce the concept of $^*$-representation topology (Definition 9.7.5, Proposition 9.7.6, Theorems 9.7.7. and 10.1.3).

\begin{theorem}\label{csillagR}
Let $A$ be a reduced $^*$-algebra and let $\Gamma_A$ be the family of all $C^*$-seminorms on $A$.
If $\tau_R$ denotes the topology defined by $\Gamma_A$, then $\tau_R$ is called the \textbf{$^*$-representation topology} on $A$, which has the following properties.
\begin{itemize}
\item[(1)]If $A$ is a $G^*$-algebra, then $\tau_R$ equals the norm topology induced by $\gamma_A$.

\item[(2)]$\tau_R$ is the strongest among all topologies $\tau$ on $A$ such that $A[\tau]$ is a $C^*$-convex algebra.

\item[(3)]$\tau_R$ is the weakest topology on $A$ which makes each $^*$-representation $\pi:A\to\mathscr{B}(\mathcal{H})$ continuous with respect to the norm topology on $\mathscr{B}(\mathcal{H})$.  

\item[(4)]A positive functional $f$ on $A$ is representable if and only if $f$ is $\tau_R$-continuous. 
\end{itemize}
\end{theorem}

The following statement shows that there is an important class of locally $C^*$-algebras such that the given topology is actually the $^*$-representation topology.

\begin{proposition}\label{frese}
If $A[\tau]$ is a Fréchet locally $m$-convex $^*$-algebra, then every $C^*$-seminorm on $A$ is $\tau$-continuous.
In particular, for Fréchet locally $C^*$-algebras $A[\tau]$ the equation $\tau=\tau_R$ holds.
\end{proposition}
\begin{proof}
Let $\sigma$ be a $C^*$-seminorm on $A$. 
By Theorem 9.5.4 in \cite{Palmer2}, there exists a $^*$-representation $\pi$ of $A$ such that $\sigma(a)=\Vert\pi(a)\Vert$ for all $a\in A$.
Now Corollary 17.2 in \cite{F} implies that $\pi$ is $\tau$-continuous, so $\sigma$ is a continuous $C^*$-seminorm on $A[\tau]$.
If, in addition, $A[\tau]$ is a locally $C^*$-algebra, property (2) of the previous Theorem \ref{csillagR} concludes that $\tau=\tau_R$.
\end{proof}

\begin{example}\label{prodCcsill}
 The $^*$-algebra $A_M:=\prod_{n\in\mathbb{N}^+}M_n(\mathbb{C})$ appeared in Example \ref{produktum}, endowed with the product topology $\tau$, is a locally $C^*$-algebra. 
In fact, $\tau$ is the $^*$-representation topology $\tau_R$.
To see this, let us consider a more general case. 
Let $(A_n,\Vert\cdot\Vert_n)_{n\in\mathbb{N}}$ be a sequence of $C^*$-algebras.
By Example 7.6 (2) in \cite{F}, the pro\-duct topology $\tau$ on $A:=\prod_{n\in\mathbb{N}}A_n$ is actually determined by the family of the $C^*$-seminorms $\Gamma=\{\sigma_F|F\subseteq\mathbb{N}, F\ $is non-void and finite$\}$, where
\begin{equation}\label{genalt}
\sigma_F:A\to\mathbb{R}_+;\ \sigma_F((a_n)_{n\in\mathbb{N}})=\sup_{n\in F}\Vert a_n\Vert_n,
\end{equation}
moreover $A[\tau]$ is a locally $C^*$-algebra.
However, the $C^*$-seminorms  
\[
\sigma_m:A\to\mathbb{R}_+;\ \sigma_m((a_n)_{n\in\mathbb{N}})=\sup_{n\leq m}\Vert a_n\Vert_n,\ (m\in\mathbb{N})
\]
also determine the topology $\tau$, hence $A[\tau]$ is a Fréchet locally $C^*$-algebra.
Now Proposition \ref{frese} implies that $\tau=\tau_R$. 
(This also concludes that the $C^*$-seminorms on $A$ have the form \eqref{genalt}.)

Consider again the locally $C^*$-algebra $A_M[\tau_R]$.
The so-called \emph{bounded part} of $A_M$ is the $^*$-subalgebra
\begin{equation*}
(A_M)_b:=\left\{(a_n)_{n\in\mathbb{N}^+}\left|a_n\in M_n(\mathbb{C}),\ \sup_{n\in\mathbb{N}}\Vert a_n\Vert_n<+\infty\right.\right\},
\end{equation*}
which is a $C^*$-algebra (cf. Question \ref{Q1} (1)) with the norm 
\[
\Vert(a_n)_{n\in\mathbb{N}^+}\Vert:=\sup_{n\in\mathbb{N}}\Vert a_n\Vert_n,
\] 
and it is $\tau_R$-dense in $A_M$ (\cite{Apostol}, Theorem 2.3; \cite{F}, Theorem 10.23).
In Example \ref{produktum}, we derived that the Lebesgue decomposition of representable positive functionals over $A_M$ is unique, because every topologically irreducible representation of $A_M$ is finite dimensional.
In contrast with this, the Lebesgue decomposition of positive functionals over the $\tau_R$-dense $C^*$-subalgebra $(A_M)_b$ is not unique, since (as we have seen in Question \ref{Q1} (1)) $(A_M)_b$ is an NGCR/antiliminal $C^*$-algebra.  
\end{example}

\begin{example}\label{hemikomp}
A Hausdorff topological space $X$ is said to be a \emph{$k$-space} if, whenever a subset $G\subseteq X$ intersects each compact $K\subseteq X$ in an open set (in the relative topology of $K$), then $G$ is open in $X$. For example, a Hausdorff, locally compact or first countable space is a $k$-space (e.g., \cite{Mi}, Appendix D).

Let $X$ be a Hausdorff completely regular $k$-space.
Denote by $A:=\mathscr{C}(X;\mathbb{C})$ the $^*$-algebra of all continuous complex valued functions on $X$, with the pointwise operations.
If $\mathfrak{K}$ stands for the set of all non-void compact subsets of $X$ then, for any $K\in\mathfrak{K}$  the mapping
\[
\sigma_K:A\to\mathbb{R}_+;\ \sigma_K(a)=\sup_{t\in K}|a(t)|,
\]
is a $C^*$-seminorm. 
According to Examples 3.10 (4) and 7.6 (3) in \cite{F}, the Hausdorff topology $\tau_{\mathfrak{K}}$ induced by the family $\Gamma=\{\sigma_K|K\in\mathfrak{K}\}$ is the \emph{topology of compact convergence} and $A[\tau_{\mathfrak{K}}]$ is a locally $C^*$-algebra.
Denote this topological $^*$-algebra by $\mathscr{C}_c(X;\mathbb{C})$.

Note here that $\tau_{\mathfrak{K}}$ is weaker than the $^*$-representation topology $\tau_R$ in general.
Indeed, let $X$ be a Hausdorff countably compact (i.e., every infinite subset of $X$ has a limit point), non-compact, completely regular space which is either locally compact or first countable (a concrete example is in the Remark after Proposition 12.2 (b) in \cite{Mi}). Then, by Proposition 12.2 (b) in \cite{Mi}, the locally $C^*$-algebra $\mathscr{C}_c(X;\mathbb{C})$ admits a discontinuous multiplicative linear functional $f$. 
Since such a functional on a locally $C^*$-algebra is hermitian (that is, $f(a^*)=\overline{f(a)}$ for any $a\in\mathscr{C}_c(X;\mathbb{C})$) by Proposition 2.5 in \cite{Apostol}, $f$ is a positive functional on $A$.
It is trivially continuous with respect to the $C^*$-seminorm
\[
\mathscr{C}_c(X;\mathbb{C})\ni a\mapsto |f(a)|\in\mathbb{R}_+,
\]
hence, according to the discusson preceding Remark \ref{represent}, $f$ is representable.
Thus $\tau_{\mathfrak{K}}=\tau_R$ is not true, since every representable positive functional is $\tau_R$-continuous (Theorem \ref{csillagR} (4)).

On the other hand, if $\mathscr{C}_c(X;\mathbb{C})$ is Fréchet, then by Proposition \ref{frese} the equation $\tau_{\mathfrak{K}}=\tau_R$ holds. Remark (ii) on page 36 in \cite{F} shows that $\mathscr{C}_c(X;\mathbb{C})$ is Fréchet if and only if $X$ is \emph{hemicompact}, e.g., $X$ is locally compact and $\sigma$-compact. 
So, there are many interesting cases when the topology of compact convergence is the $^*$-representation topology.
\end{example}

Other interesting and important examples can be found in, e.g., \cite{Apostol}, \cite{Bhatt}, \cite{F}, \cite{Inoue}, \cite{Phil}.

We introduce another relevant concept, the so-called enveloping locally $C^*$-algebra (\cite{F}, Section 18 in Chapter IV; \cite{Bhatt}, Section 2).
It can be defined in a more general setting (in the case of not necessarily $m$-convex or reduced $^*$-algebras), but for simplicity we make extra assumptions.

\begin{definition}\label{fedolok}
Let $A[\tau]$ be a Hausdorff locally $m$-convex $^*$-algebra.
Suppose that the family $\Gamma_{C^*}$ of the $\tau$-continuous $C^*$-seminorms separate the points of $A$.
Then the completion $\widetilde{A}[\tau_{\widetilde{\Gamma_{C^*}}}]$ of the $C^*$-convex algebra $A[\tau_{\Gamma_{C^*}}]$ is the \emph{enveloping locally $C^*$-algebra} of $A[\tau]$, and it is denoted by $E(A)[\tau_E]$ (or simply $E(A)$, if there is no confusion).
\end{definition}

\begin{example}If $A[\tau]$ is a $C^*$-convex algebra, then $E(A)$ is just the completion of $A[\tau]$ (in particular, if $A[\tau]$ is complete, then $E(A)$ is $A[\tau]$ itself).
For instance, the $^*$-algebra of polynomials in Example \ref{polialg}, endowed with the topology of compact convergence $\tau_{\mathfrak{K}}$ (which equals with $^*$-representation topology) is a $C^*$-convex algebra.
Its completion/enveloping locally $C^*$-algebra is $\mathscr{C}_c(\mathbb{R};\mathbb{C})$. 

For a non-$C^*$-convex example, consider the algebra (with pointwise operations) $\mathcal{O}(\mathbb{C})$ of all analytic functions defined on the entire plane $\mathbb{C}$.
Similar to Example \ref{Acsillag} (1), define the involution by
\[
^*:\mathcal{O}(\mathbb{C})\to \mathcal{O}(\mathbb{C});\ a^*(z)=\overline{a(\overline{z})}\ (z\in\mathbb{C}).
\]  
Then the family $\Gamma$ of the submultiplicative norms
\[
\sigma_n:A\to\mathbb{R}_+;\ \sigma_n(a)=\sup_{|z|\leq n}|a(z)|\ (n\in\mathbb{N}^+)
\]
determines a Hausdorff locally $m$-convex topology on $\mathcal{O}(\mathbb{C})$, and $\mathcal{O}(\mathbb{C})[\tau_{\Gamma}]$ is a Fréchet locally $m$-convex $^*$-algebra (\cite{F}, Examples 2.4 (5) and 3.2 (4)), which is not a locally $C^*$-algebra.
The enveloping locally $C^*$-algebra of $\mathcal{O}(\mathbb{C})[\tau_{\Gamma}]$ is $\mathscr{C}_c(\mathbb{R};\mathbb{C})$  (\cite{F}, 18.9 (1)).

Note that for Fréchet locally $m$-convex $^*$-algebras with a separating system of continuous $C^*$-seminorms, the topology of the enveloping locally $C^*$-algebra is the $^*$-representation topology.
Indeed, $E(A)$ is metrizable (\cite{Bhatt}, the paragraph preceding to Lemma 2.5), thus Proposition \ref{frese} applies. 
\end{example}

\begin{remark}\label{conlc}
The concept of the enveloping locally $C^*$-algebra does not necessarily agree with the enveloping $C^*$-algebra (Definition \ref{fedoCalg}).
For instance, the disc $^*$-algebra $A(\mathbb{D})$ in Example \ref{Acsillag} (1) is a reduced Banach $^*$-algebra (hence a $G^*$-algebra) and its enveloping $C^*$-algebra $C^*(A(\mathbb{D}))$ is $\mathscr{C}([-1,1];\mathbb{C})$ equipped with the supremum norm. 
On the other hand, if we endow $A(\mathbb{D})$ with the topology $\tau$ induced by the pre-$C^*$-norm
\[
\sigma_{[0,1]}:A(\mathbb{D})\to A(\mathbb{D});\ \sigma_{[0,1]}(a):=\sup_{t\in[0,1]}|a(t)|,
\]
then $A[\tau]$ is a $C^*$-convex algebra and $E(A)=\mathscr{C}([0,1];\mathbb{C})$ equipped with the topology induced by the supremum norm on $[0,1]$. 

The difference is that $E(A)$ depends on the initial topology of $A[\tau]$, while (when $A$ is a $G^*$-algebra) $C^*(A)$  is the completion with respect to the concrete pre-$C^*$-norm $\gamma_A$.
The latter $C^*$-algebra has the "same" $^*$-representation theory as that of $A$ (discussion after Definition \ref{fedoCalg}), while $E(A)$ has the "same" continuous $^*$-representation theory as that of $A[\tau]$ (e.g.: \cite{F}, page 207, Introduction to Chapter IV and Theorem 18.8). 
However, in the case of the $^*$-representation topology, $E(A)=C^*(A)$ holds for $G^*$-algebras (see Proposition \ref{bhatt} below).  
\end{remark}

We answer Question \ref{quest} after the following helpful statement.

\begin{proposition}\label{bhatt}
Let $A$ be a reduced $^*$-algebra and let $\tau_R$ be the $^*$-rep\-re\-sen\-ta\-tion topology on $A$.
The following statements are equivalent.
\begin{itemize}
\item[(i)]$A$ is a $G^*$-algebra.
\item[(ii)]The enveloping locally $C^*$-algebra $E(A)$ of the $C^*$-convex algebra $A[\tau_R]$ is topologically $^*$-isomorphic to a $C^*$-algebra.
\end{itemize}
If the conditions are fulfilled, then $E(A)=C^*(A)$. 
\end{proposition}
\begin{proof}
(i)$\Rightarrow$(ii):
If $A$ is a $G^*$-algebra, then by Theorem \ref{csillagR} (1), $\tau_R$ is norm topology induced by $\gamma_A$.
Thus (ii) and $E(A)=C^*(A)$ are immediately follow.

(ii)$\Rightarrow$(i):
Assume that the topology of $E(A)$ is determined by the $C^*$-norm $\sigma_{E(A)}$ and let $\sigma$ be a $C^*$-seminorm on $A$. 
We prove for any $a\in A$ that 
\[
\sigma(a)\leq\sigma_{E(A)}(a)
\]
 holds true and hence, consequently $\sigma_{E(A)}|_A=\gamma_A$ and $A$ is a $G^*$-algebra.

By Theorem 9.5.4 in \cite{Palmer2}, there exists a $^*$-representation $\pi$ of $A$ such that $\sigma(a)=\Vert\pi(a)\Vert$ for all $a\in A$.
Since $\pi$ is $\tau_R$-continuous (Theorem \ref{csillagR} (3)), it extends to a (continuous) $^*$-representation $\widetilde{\pi}$ of the enveloping locally $C^*$-algebra $E(A)$.
But it is well known that each $^*$-representation of a $C^*$-algebra is a contraction (\cite{Palmer2}, Theorem 9.5.12 (e)), that is, 
\[
\Vert\widetilde{\pi}(a)\Vert\leq\sigma_{E(A)}(a)
\]
for every $a\in E(A)$. 
Now the conclusion follows.
\end{proof}

Our last comment answers Question \ref{quest} and explains that the most general setting for the Lebesgue decomposition theory of all representable positive functionals over reduced $^*$-algebras is in fact the case of locally $C^*$-algebras with the $^*$-representation topology.  

Consider a reduced $^*$-algebra $A$ and equip it with the $^*$-representation to\-po\-lo\-gy $\tau_R$.
Let $E(A)$ be the completion/enveloping locally $C^*$-algebra of $A[\tau_R]$.
Then the topology $\tau_E$ of $E(A)$ is its $^*$-representation topology.
Since the representable positive functionals on $A$ are precisely the $\tau_R$-continuous positive functionals (Theorem \ref{csillagR} (4)), their unique $\tau_E$-continuous linear extensions to $E(A)$ are positive and representable.
Moreover, if $f$ and $g$ are representable positive functionals on $A$ such that $f\leq g$, then their extensions $\widetilde{f}$ and $\widetilde{g}$ satisfy $\widetilde{f}\leq\widetilde{g}$.
By the definitions of absolute continuity and singularity (Definitions \ref{absdef} and \ref{singdef}) and the Lebesgue decomposition (Theorem \ref{Lebesgue}), it is easy to see that the following are true:
\begin{itemize}
\item $f\ll g \Leftrightarrow \widetilde{f}\ll\widetilde{g}$ and $f\perp g \Leftrightarrow \widetilde{f}\perp\widetilde{g}$;
\item If $f=f_r+f_s$ is the Lebesgue decomposition with respect to $g$ and $\widetilde{f}=(\widetilde{f})_r+(\widetilde{f})_s$ is the Lebesgue decomposition with respect to $\widetilde{g}$, then $(\widetilde{f})_r=\widetilde{f_r}$ and $(\widetilde{f})_s=\widetilde{f_s}$.
Furthermore, since the $^*$-representation theory of $A$ is in a bijective correspondence with the $^*$-representation theory of $E(A)$, the uniqueness holds for the decomposition over $A$ if and only if uniqueness holds over $E(A)$ (Theorem \ref{mainnew}).
\end{itemize}
By Proposition \ref{bhatt}, $E(A)$ is a $C^*$-algebra iff $A$ is a $G^*$-algebra.
Hence the discussion above completely answers Question \ref{quest} as well.

\begin{acknowledgements} The authors thank the anonymous referee the valuable suggestions that improved the presentation of the paper.
\end{acknowledgements}

%\section*{Statements and Declarations}

%\begin{itemize}
%\item The authors have no relevant financial or non-financial interests to disclose.

%\item The authors have no competing interests to declare that are relevant to the content of this article.

%\item All authors certify that they have no affiliations with or involvement in any organization or entity with any financial interest or non-financial interest in the subject matter or materials discussed in this manuscript.

%\item The authors have no financial or proprietary interests in any material discussed in this article.
%\end{itemize}


\end{document}